\documentclass[11pt]{amsart}
\usepackage{amsbsy}
\usepackage{amsmath}
\usepackage{amsfonts}
\usepackage{graphicx,epsfig,subfigure,psfrag}
\usepackage{color}

\usepackage{changes}

\begin{document}

%%%%%%%%%%%%%%%%%%%%%%%%%%%%%%%%%%%%%%%%%%%%%%%%%%%%%%%%%%%%%%%%%%%%
% Theorem, definition, lemma, proposition, corollary and proof
%%%%%%%%%%%%%%%%%%%%%%%%%%%%%%%%%%%%%%%%%%%%%%%%%%%%%%%%%%%%%%%%%%%%
%%%%%%%%%%%%%%%%%%%%%%%%%%%%%%%%%%%%%%%%%%%%%%%%
\newtheorem{theorem}{Theorem}[section]
\newtheorem{proposition}[theorem]{Proposition}%[section]
\newtheorem{lemma}[theorem]{Lemma}%[section]
\newtheorem{corollary}[theorem]{Corollary}%[section]%%
\newtheorem{definition}[theorem]{Definition}%[section]
\newtheorem{remark}[theorem]{Remark}%[section]
%%%%%%%%%%%%%%%%%%%%%%%%%%%%%%%%%%%
%%%%%%%%%%%%%%%%%%%%%%%%%%%%%%%%%%%%%%%%%%%%%%                  NEW
%%\newcommand{\be}{\begin{equation}}
%%\newcommand{\ee}{\end{equation}}
%%%%%%%%%%%%%%%%%%%%%%%%%%%%%%%%%%%%%%%%%%%%%%%
%%%%%%%%%%%%%%%%%%%%%%%%%%%%%%%%%%%%%%%%%%%%
\newcommand{\tex}{\textstyle}
%%\DeclareMathOperator{\Ind}{Ind}
%% \DeclareMathOperator{\Card}{Card}
%% \DeclareMathOperator{\Deg}{Deg}
%% \DeclareMathOperator{\dist}{dist}
%% \DeclareMathOperator{\Signature}{Signature}
%% \DeclareMathOperator{\MC}{MC}
%% \DeclareMathOperator{\sign}{sign}
%%  \DeclareMathOperator{\Int}{int}
%%%%%%%%%%%%%%%%%%%%%%%%%%%%%%%%%%%%%%%%%%%%
%%%%%%%%%%%%%%%%%%%%%%%%%%%%%%%%%%%%%%%%%%%%
\numberwithin{equation}{section}
% \numberwithin{theorem}{section}
%\numberwithin{proposition}{section} \numberwithin{lemma}{section}
%\numberwithin{corollary}{section}
%\numberwithin{definition}{section}
% \numberwithin{remark}{section}
%%%%%%%%%%%%%%%%%%%%%%%%%%%%%%%%%%%%%%%%%%%%
%%%%%%%%%%%%%%%%%%%%%%%%%%%%%%%%%%%%%%%%%%%%
\newcommand{\ren}{\mathbb{R}^N}
\newcommand{\re}{\mathbb{R}}
\newcommand{\n}{\nabla}
\newcommand{\p}{\partial}
\newcommand{\iy}{\infty}
\newcommand{\pa}{\partial}
\newcommand{\fp}{\noindent}
\newcommand{\ms}{\medskip\vskip-.1cm}
\newcommand{\mpb}{\medskip}
%%%%%%%%%%%%%%%%%%%%%%%%%%%%%%%%%%%%%%%%%%%%%%%%%
\newcommand{\AAA}{{\bf A}}
\newcommand{\BB}{{\bf B}}
\newcommand{\CC}{{\bf C}}
\newcommand{\DD}{{\bf D}}
\newcommand{\EE}{{\bf E}}
\newcommand{\FF}{{\bf F}}
\newcommand{\GG}{{\bf G}}
\newcommand{\oo}{{\mathbf \omega}}
\newcommand{\Am}{{\bf A}_{2m}}
\newcommand{\CCC}{{\mathbf  C}}
\newcommand{\II}{{\mathrm{Im}}\,}
\newcommand{\RR}{{\mathrm{Re}}\,}
\newcommand{\eee}{{\mathrm  e}}
%%%%%%%%%%%%%%%%%%%%%%%%%%%%%%%%%%%%%%%%%%%%%%%%%%%%%%%%%%%%%%%%%%%%%%% L^2\rho...
\newcommand{\LL}{L^2_\rho(\ren)}
\newcommand{\LLL}{L^2_{\rho^*}(\ren)}
%%%%%%%%%%%%%%%%%%%%%%%%%%%%%%%%%%
%%%%%%%%%%%%%%%%%%%%%%%%%%%%%%%%%%%%%%%%%%%%%%%%%%%%
\renewcommand{\a}{\alpha}
\renewcommand{\b}{\beta}
\newcommand{\g}{\gamma}
\newcommand{\G}{\Gamma}
\renewcommand{\d}{\delta}
\newcommand{\D}{\Delta}
\newcommand{\e}{\varepsilon}
\newcommand{\var}{\varphi}
\newcommand{\lll}{\l}
\renewcommand{\l}{\lambda}
\renewcommand{\o}{\omega}
\renewcommand{\O}{\Omega}
\newcommand{\s}{\sigma}
\renewcommand{\t}{\tau}
\renewcommand{\th}{\theta}
\newcommand{\z}{\zeta}
\newcommand{\wx}{\widetilde x}
\newcommand{\wt}{\widetilde t}
\newcommand{\noi}{\noindent}
 %%%%%%%%%%%%%%%%%%%%%%%%%%%%%%%%%%%%%%%%%%%
\newcommand{\uu}{{\bf u}}
\newcommand{\xx}{{\bf x}}
\newcommand{\yy}{{\bf y}}
\newcommand{\zz}{{\bf z}}
\newcommand{\aaa}{{\bf a}}
\newcommand{\cc}{{\bf c}}
\newcommand{\jj}{{\bf j}}
\newcommand{\ggg}{{\bf g}}
\newcommand{\UU}{{\bf U}}
\newcommand{\YY}{{\bf Y}}
\newcommand{\HH}{{\bf H}}
\newcommand{\GGG}{{\bf G}}
\newcommand{\VV}{{\bf V}}
\newcommand{\ww}{{\bf w}}
\newcommand{\vv}{{\bf v}}
\newcommand{\hh}{{\bf h}}
\newcommand{\di}{{\rm div}\,}
\newcommand{\ii}{{\rm i}\,}
\def\I{{\mathbf{I}}}
%%%%%%%%%%%%%%%%%%%%%%%%%%%%%%%%%%
%%%%%%%%%%%%%%%%%%%%%%%%%%%%%%%%%%%%%   VAG, NEW
\newcommand{\inA}{\quad \mbox{in} \quad \ren \times \re_+}
\newcommand{\inB}{\quad \mbox{in} \quad}
\newcommand{\inC}{\quad \mbox{in} \quad \re \times \re_+}
\newcommand{\inD}{\quad \mbox{in} \quad \re}
\newcommand{\forA}{\quad \mbox{for} \quad}
\newcommand{\whereA}{,\quad \mbox{where} \quad}
\newcommand{\asA}{\quad \mbox{as} \quad}
\newcommand{\andA}{\quad \mbox{and} \quad}
\newcommand{\withA}{,\quad \mbox{with} \quad}
\newcommand{\orA}{,\quad \mbox{or} \quad}
\newcommand{\atA}{\quad \mbox{at} \quad}
\newcommand{\onA}{\quad \mbox{on} \quad}
\newcommand{\ef}{\eqref}
\newcommand{\mc}{\mathcal}
\newcommand{\mf}{\mathfrak}

\newcommand{\ssk}{\smallskip}
\newcommand{\LongA}{\quad \Longrightarrow \quad}
%%%%%%%%%%%%%%%%%%%%%%%%%%%%%%%%
%%%%%%%%%%%%%%%%%%%%%%%%%%%%%%%%%%
\def\com#1{\fbox{\parbox{6in}{\texttt{#1}}}}
%%%%%%%%%%%%%%%%%%%%%%%%%%%%%%%%%%
%%%%%%%%%%%%%%%%%%% From Paper1
\def\N{{\mathbb N}}
\def\A{{\cal A}}
\newcommand{\de}{\,d}
\newcommand{\eps}{\varepsilon}
\newcommand{\be}{\begin{equation}}
\newcommand{\ee}{\end{equation}}
\newcommand{\spt}{{\mbox spt}}
\newcommand{\ind}{{\mbox ind}}
\newcommand{\supp}{{\mbox supp}}
\newcommand{\dip}{\displaystyle}
\newcommand{\prt}{\partial}
\renewcommand{\theequation}{\thesection.\arabic{equation}}
\renewcommand{\baselinestretch}{1.1}
%%%%%%%%%%%%%%%%%%%%%%%%%%%%%%%%%%%%%%%%%%%%%%%
\newcommand{\Dm}{(-\D)^m}

%%%%%%%%%%%%%%%%%%%%%%%%% 
\title
%%%%%[Positivity]
%%%%%%%%%%%%%%%%%%%%%%%%%
{\bf A stationary population model with an interior interface-type boundary}

\author{Pablo \'Alvarez-Caudevilla and Cristina Br\"{a}ndle}

\address{Universidad Carlos III de Madrid,
Av. Universidad 30, 28911-Legan\'es, Spain} \email{pacaudev@math.uc3m.es}

\address{Universidad Carlos III de Madrid,
Av. Universidad 30, 28911-Legan\'es, Spain} \email{cbrandle@math.uc3m.es}


\thanks{This paper has been partially supported by Ministry of Economy and Competitiveness of Spain under research project PID2019-106122GB-I00}


\date{\today}

%\parskip5pt
%%%%%%%%%%%%%%%%%%%%%%%%%%%%%%%%%%%%%%%%%%%%%%%%%%%%%%%%

\begin{abstract}
We propose a stationary system that might be regarded as a migration model of some population abandoning their original place of abode and becoming part of another population,
once they reach the interface boundary. To do so, we show a model where each population follows a logistic 
equation in their own environment while assuming spatial heterogeneities.
Moreover, both populations are coupled through the common boundary, which acts as a permeable membrane on which their flow moves in and out.
The main goal we face in this work will be
to describe the precise interplay between the stationary solutions with respect to the parameters involved in the problem, in particular the growth rate of the populations and the coupling parameter
involved on the boundary where the interchange of flux is taking place.
\end{abstract}

%%%%%%%%%%%%%%%%%%%%%%%%%%%
\maketitle

%%%%%%%%%%%%%%%%%%%%%%%%%%%%%%%%%%%%%%%%%%%%%%%%
\section{Introduction}
 \label{S1}

In this work we analyse the existence of stationary solutions of two different species living separately  in two regions of a heterogeneous environment
and having an interaction through a permeable membrane,  the common boundary of the two regions.

{To be more specific, }the problem we have in mind is a reaction-diffusion model with spatially heterogeneous coefficients that provides us with the stationary {non-negative} solutions for an evolution model
of two species, that live separately in two subdomains $\Omega_1$ and $\Omega_2$ of $\Omega\subset\mathbb{R}^N$, see Figure~\ref{fig:domain} and Section~\ref{sect:preliminaries} for a precise description of the domain. So that, if we denote the density of these populations
by $u_1, u_2$, we have that $u_i>0$ in $\Omega_i$ and $u_i=0$ in $\Omega\setminus\Omega_i$,  with $i=1,2$. The  model that describes their stationary behaviour is shown as
\begin{equation}
\label{main}
\left\{ \begin{array}{llll}
-\Delta u_1(x)=\lambda m_1(x) u_1(x)-a_1(x) u_1(x)^p,&\quad \text{ in } \Omega_1,\quad & 
 \\[6pt]
-\Delta u_2(x)= \lambda m_2(x) u_2(x)-a_2(x) u_2(x)^p,&\quad \text{ in } \Omega_2,
\quad & 
\end{array}\right.
\end{equation}
where $\lambda$ is a real parameter and $p>1$. The weights $m_i$ are assumed to be two positive $L^\infty(\Omega_i)$ functions, so that   $\lambda m_i(x)$
stands for the intrinsic growth rate of $u_i$ respectively.
The coefficients $a_i(x)$
 are two non-negative H\"{o}lder continuous functions in $\mathcal{C}^{0,\eta}(\overline{\Omega_i})$, for some $\eta\in (0,1]$, that  measure the crowding effects of
the populations in the corresponding subdomains of $\Omega_i$.
In order to take into account these effects, we define the following subsets, related to the positivity of the potentials $a_1$ and $a_2$,
\begin{equation}
  \label{eq:omega_a}
  \Omega_0^{a_1}:=\{x\in\Omega_1\;:\;a_1(x)=0\}\quad \text{and}\quad \Omega_0^{a_2}:=\{x\in\Omega_2\;:\;a_2(x)=0\}
\end{equation}
so that   $\overline \Omega_0^{a_1}\subset  \Omega_1$,   and $\overline \Omega_0^{a_2} \subset\Omega_2$, while $a_i$ remain positive in the rest of the subdomain $\Omega_i$.

\begin{figure}[ht!]
\includegraphics[width=8cm]{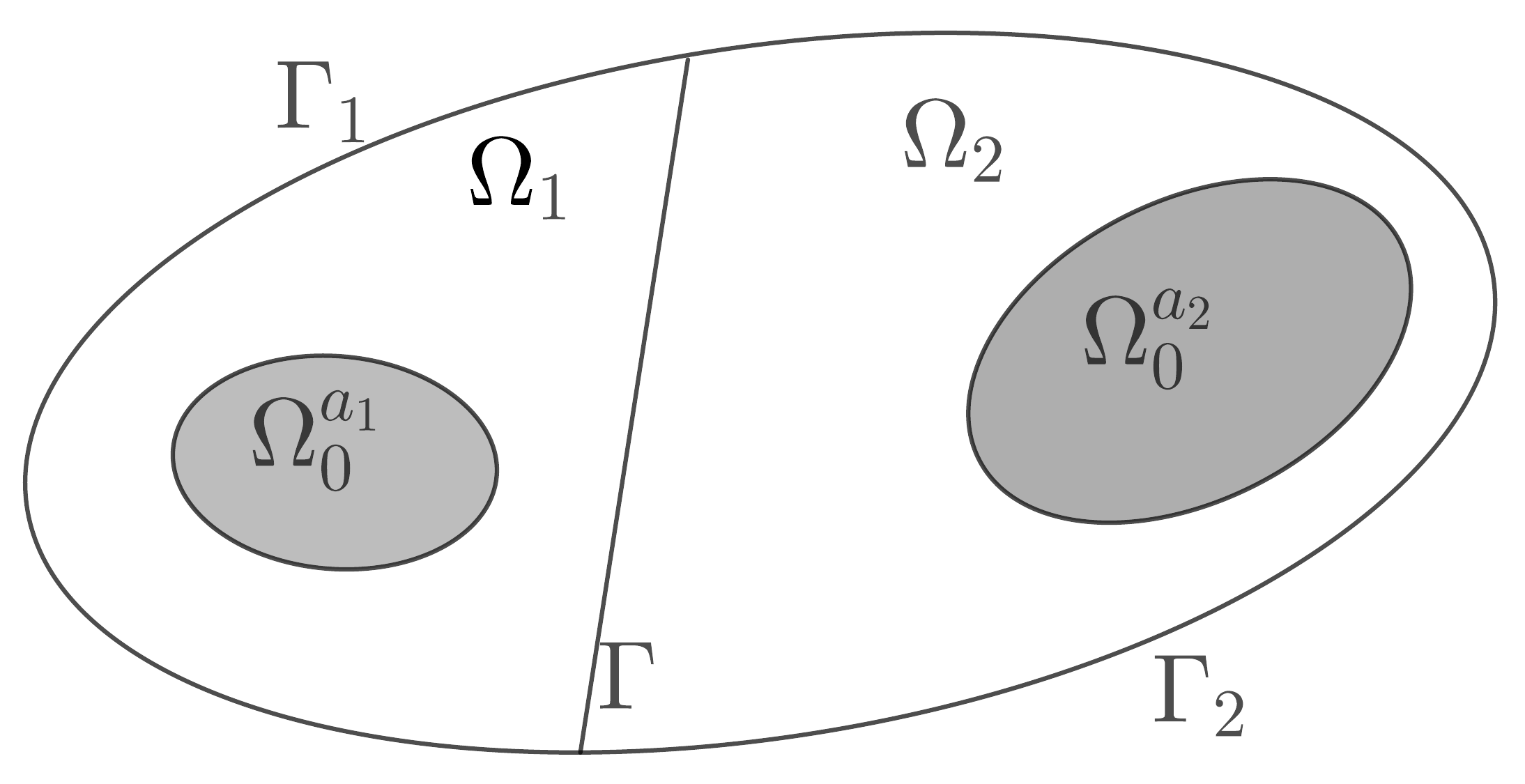}
\caption{A possible configuration of the domain.}
\label{fig:domain}
\end{figure}

The  two populations only interact through the interior boundary, also called interface or membrane 
$\Gamma$. Moreover, we will consider a hostile outside region. Hence, the boundary conditions read as follows
\begin{equation}
\label{opbound}
\left\{\begin{array}{ll} \dfrac{\p u_1}{\p {\bf n_1}}= \dfrac{\p u_2}{\p {\bf n_1}}= \mu(u_2- u_1),& \quad \text{on } \Gamma,\\[10pt]
u_i =0,  & \text{\quad on } \Gamma_i, \quad \text{with}\quad  i=1,2,\\
\end{array}\right.
\end{equation}
where ${\bf n_1}$ is the unitary outward normal vector of $\Gamma$, which points outward with respect to $\Omega_1$  and $\mu>0$ is a real parameter, the membrane permeability constant.
{Note that the intersections between both manifolds $\Gamma_i$ and $\Gamma$ satisfy both boundary conditions. Hence, at these points we have the continuity between both sides and
the set $\Gamma$ could be assumed to be closed if we include these intersection points, in case of having
the standard configuration shown in
Figure~\ref{fig:domain}.

The major novelty that we address in this paper stems from the boundary conditions. Indeed, besides the usual Dirichlet data on the boundary of $\Omega$, we also consider an interior interface condition. Our aim is to describe  the behaviour of the solutions for that interface problem.   The main results that we show, see Sections~\ref{sec:nondegenerate} and~\ref{sec:degenerate}, state the existence and behaviour of solutions to problem~\eqref{main}--\eqref{opbound} considering two different configurations for the potentials $a_i$.
First, we will deal with the non-degenerate case, in which $a_i$ are positive everywhere. We obtain that there exist positive solutions to~\eqref{main}--\eqref{opbound} 
if the parameter $\lambda$ is bigger than a certain  $\lambda^*$, while below that value $\lambda^*$ we just have the trivial solution everywhere.
The situation changes radically in the degenerate case, in which  we assume two vanishing domains, $\Omega_0^{a_i}$, inside the subdomains $\Omega_i$, where the coefficients $a_i$
are neglected. In this case, problem \eqref{main}--\eqref{opbound} has positive solutions if the parameter $\lambda$ is bounded between two precise constants $\lambda^*<\lambda<\lambda_\infty$. 
As in the previous case, if $\lambda<\lambda^*$ one just finds the trivial solution. If $\lambda\geq \lambda_\infty$ we get a blow-up behaviour in uniformly compact subsets of 
$\Omega_0^{a_i}$. In the complement of these blow-up regions we find a steady-state solution
that blows-up at the border of the regions. We notice that, even though the  behaviour we obtain here is similar to a logistic model with spatial heterogeneities and vanishing subdomains (cf.~\cite{AC-LG2}), our model has an important geometrical issue which
makes the results more constrained to some conditions.
These geometrical aspects give richer situations and allow us to study a novel  branch of problems that aim to model situations in which interior membranes appear, a circumstance which has been little studied from a
rigorous point of view.

Furthermore, many  problems that arise in biology, physics and/or medical science can be modelled by considering two subdomains with a common interior membrane that allows
permeation from one domain to the other one.  This interface boundary condition, also called the Kedem-Katchalsky interface condition,~\cite{KedemK}, was introduced in 1961 in a thermodynamic setting.
These conditions, describing the flow through the membrane, are compatible with mass
conservation lead to flux continuity, and energy dissipation, that gives us that the flux is proportional to the jump of the function through the
membrane with proportionality coefficient $\mu$; see~\cite{CiavolellaPaerthame} for further details.

From a biological point of view, the use of Kedem-Katchalsky interface conditions goes back to~\cite{Quarteroni}, where the authors studied
the dynamics of the solute in the vessel and in the arterial wall. Since then,  biological applications of membrane problems have increased and have been used  to describe phenomena such as tumour invasion, transport of molecules through the
cell/nucleus membrane, cell polarisation and cell division or  genetics, see~\cite{CiavolellaPaerthame} and references therein.

As far as the study of the mathematical properties is concerned, articles that include a geographical barrier are relatively few. We were only aware of a few works: In~\cite{WangSu} a semi-linear parabolic problem is considered. They investigate the effects of the barrier on the global
dynamics and on the existence, stability, and profile of spatially non-constant equilibria. A similar problem was previously considered in~\cite{Chen2001} and~\cite{Chen2006}.  His analysis was mainly in the framework
of the existence of weak solutions of parabolic and elliptic differential equations with barrier boundary conditions. The author establishes a new comparison principle, the global existence of solutions, and sufficient conditions of stability and instability of equilibria. He shows, in particular, that the stability of equilibrium changes as the barrier permeability changes through a critical value. Very recently, Ciavolella and Perthame~\cite{CiavolellaPaerthame} and~\cite{Ciavolella}  adapted the well-known $L^1$-theory for parabolic reaction-diffusion systems to the membrane boundary conditions case and proved several regularity results. Finally, in~\cite{Suarezetal} a
semi-linear elliptic interface problem is studied. The authors analyse the existence and uniqueness of positive solutions, under the assumption that  the interface condition in not symmetric, allowing different in and out fluxes of the domain. 
Although we assume symmetry on the interface our results are also valid for non-symmetric conditions. It is worth mentioning that all these references deal with problems that do not consider refuges (that is $\Omega_0^{a_i}$ are empty).

\noindent{\sc Organisation of the paper.} --- In Section~\ref{sect:preliminaries} we  go through the model we are dealing with and fix the notation that we are going to use throughout the paper. Section~\ref{subsect:eigenvalue_problem} is devoted to the study of two auxiliary linear problems that take into account the interface condition. In Section~\ref{Section_asym}  we describe the limit behaviour of a parameter dependent linear elliptic eigenvalue problem that will be crucial to establish the main results of the paper, when applying the method of sub and supersolutions for the existence of solutions. Finally, we state and prove the main results of the paper in Sections~\ref{sec:nondegenerate} and~\ref{sec:degenerate}. Section~\ref{sect:conclusions} is devoted to presenting  the conclusions of this paper and to establish some open problems.

%%%%%%%%%%%%%%%%%%%%%%%%%%%%%%%%%%%%%%%%%%
\section{Preliminaries}
\label{sect:preliminaries}
%%%%%%%%%%%%%%%%%%%%%%%%%%%%%%%%%%%%%%%%%%

The aim of this section is to comment on the special characteristics of the model. In particular, we describe it both from a biological and mathematical point of view. We also describe the notation that we are going to use. 

\subsection{The model}

The equations in~\eqref{main} arise occur in population dynamics as a stationary model of reaction-diffusion equations in which the individuals of a group live in a domain where neither competition or cooperation with other species exists. The model equation
$$
-\Delta u(x) =\lambda m(x) u(x)-a(x)u(x)^p
$$
assuming spatial heterogeneities was first mathematically analysed by Cantrell and Cosner in their seminal paper~\cite{CantrellCosner1991}, with $p=2$. Here the population grows following a logistic (when $a>0$) or Malthusian behaviour (in the regions where $a=0$), $\lambda m(x)$ stands for the intrinsic growth rate of the population, since the function $m(x)$ is going to be considered strictly positive, at any point $x$, and $\lambda$ will be a positive parameter. Moreover, $a(x)$ is the so-called intra-specific competition. Therefore, in the regions where $a=0$, the refuges or regions
 with unlimited resources, the population grows without limits.

Hence, the model we consider here could represent the migration or relation of  two groups of the previous population living in different regions of a common habitat, but having some interaction between them, through an interchange of flux on the common boundary, $\Gamma$, {that can be seen as a natural or geographical barrier}.
For the transmission condition
$$\dfrac{\p u_1}{\p {\bf n_1}}= \dfrac{\p u_2}{\p {\bf n_1}}= \mu(u_2- u_1)$$
the parameter $\mu$ measures  the strength of the membrane.
The flux of individuals remains continuous being proportional to the change of density across
the barrier by a rate $1/\mu$. Therefore,
the smaller $\mu$ is the stronger the barrier is.
And in fact, at the limit, when $\mu\to 0$, there is no transmission across $\Gamma$. The limit problem consists of two populations living separately in two different domains with homogenous mixed boundary conditions, Dirichlet on $\Gamma_i$ and Neumann on $\Gamma$, see~\cite{Santi} for properties of this type of problems.

If only one of the populations is zero on $\Gamma$, say $u_1$, then, again, the problem for $u_2$ becomes uncoupled, {with mixed homogenous boundary conditions, Robin on $\Gamma$ and Dirichlet on $\Gamma_2$. Clearly, we can understand  the problem that defines the behaviour of $u_1$ as a problem with non-homogeneous mixed boundary data in $\Omega_1$.} Moreover, this would also allow the existence of semi-trivial solutions of the type  $(u_1,0)^T$ having now $u_1$ homogeneous boundary conditions  on $\partial\Omega_1$.

%%%%%%%%%%%%%%%%%%%%%%%%%%%%%%%%%%%%
\subsection{Basic notation and functional setting.}
\label{subsect:functional_setting}

Though one of the aims we have in mind consists of understanding this problem as a whole system, in vectorial logistic form in the whole domain $\Omega$,
depending on the results and proofs we are dealing with, sometimes it will be more convenient to work with a two equation point of view. For the sake of completeness, we think that it is worth writing this short subsection to state the notation we are going to use throughout the paper.

\noindent{\it Domains and geometry}: To fix the geometric structure of the problem let $\Omega\subset \mathbb{R}^N$, with $N\geq 2$
 be a bounded, connected domain.  We split $\Omega$ into two non-empty, regular sets, $\Omega_1$ and $\Omega_2$, such that
 $\overline{\Omega}=\overline{\Omega}_1\cup \overline{\Omega}_2$. We will denote by $\Gamma$ the  inside
 boundary, $\Gamma=\overline{\Omega}_1\cap \overline{\Omega}_2$
 and by $\Gamma_1:=\partial\Omega_1\setminus\Gamma$ and $\Gamma_2:=\partial\Omega_2\setminus\Gamma$ the  outside  boundaries.
 It is worth noting that in Figure~\ref{fig:domain} we are assuming that $\Gamma_1$ and $\Gamma_2$ are non-empty. However, it would be also possible to have, say $\Gamma_1=\emptyset$ and $\Gamma=\partial\Omega_1$ which in turn would mean $\Omega_1\subset\Omega_2$ (see~\cite{Pfluger1996} or~\cite{Suarezetal}). This configuration, though possible and easily applicable, is not considered here. 
 We also assume that these domains, as well as all the possible subsets appearing throughout the paper, are as regular as necessary; let us say with Lipschitz boundary, which again will play a role
in obtaining some of the results.

 We will use the subscript $1$, respectively $2$, to denote objects (functions, parameters...) defined in $\Omega_1$ (resp. $\Omega_2$) and each time we write the subscript $i$ we mean $i=1,2$ and we will not mention it anymore.

\noindent{\it Vectors}: In order to simplify notation, we will use capital letters to denote the vector corresponding to a pair of functions (typically solutions of a system), where the first entry is defined in $\Omega_1$, or defined equal to zero in $\Omega_2$, and the second one in $\Omega_2$. Thus,  for example, we will write $U$ to denote the vector $(u_1,u_2)^T$, where $u_i:\Omega_i\to\mathbb{R}$ and $U(x)$ will stand for $(u_1(x), 0)^T$ if $x\in\Omega_1$ and $(0,u_2(x))^T$ if $x\in\Omega_2$.
{By abuse of notation we will use  $\nabla U$ to denote  $(\nabla u_1,\nabla u_2)^T$}. Also, we will write $\Phi>0$ if $\phi_1>0 $ in $\Omega_1$ and $\phi_2>0$ in $\Omega_2$ and { $\Phi\succeq 0$ if we allow one of the components to be identically {zero} in its domain, while the other one is positive.}

\noindent{\it Matrices}: All the matrices that appear in this paper come from a two equation framework. Hence, they are $2\times 2$ and in most of the cases, diagonal.

We  will keep the letter $\mathbf{L}$ to denote the \lq\lq laplacian diagonal matrix,\rq\rq\,  $\mathbf{I}$ will stand for the $2\times 2$ identity matrix  and we will denote other matrices using capital boldface letters, for instance
\begin{equation*}
  \label{notation}
\mathbf{L}=  \left(\begin{array}{cc} -\D  & 0\\
     0 & -\D \end{array}\right),\qquad
\mathbf{F}=\left(\begin{array}{cc} f_1  & 0\\
     0 & f_2 \end{array}\right)
\end{equation*}
 where $f_i$ will be functions. We will denote $\mathbf{L}+\mathbf{F}$ as $\mathbf{L_F}$. In Section~\ref{sec:nondegenerate} we will deal with the inverse of the laplacian. With this aim, we will also define  the diagonal matrix
\begin{equation}
\label{inverse.lap}{ \mathbf{L}^{-1}=\left(\begin{array}{cc} (-\Delta)^{-1} & 0\\ 0 & (-\Delta)^{-1}\end{array}\right)}.
\end{equation}

 As before, we will write $\mathbf{C}>0$ (or $\geq$) to denote a matrix with positive (non-negative) entries and we will write $\mathbf{C}>\mathbf{D}$ if $\mathbf{C}-\mathbf{D}>0$ (respectively $<$ or $\leq$, $\geq$). Moreover, $\mathbf{C}^+$ will stand for the positive part of the matrix in the sense that we take the positive part of each entry (which are usually functions) in the matrix.

\noindent {\it Integrals}: To simplify notation and if no confusion arises, as is commonly occurs in the literature, we will avoid writing differentials on the integrals. It will be understood that integrals in $\Omega$ or $\Omega_i$ require $dx$ and integrals over boundaries,  for instance $\Gamma$ or $\partial\Omega_i$, have a surface differential such as $ds$.

\noindent{\it Functional spaces}: In order to deal with problem~\eqref{main}--\eqref{opbound} and the related ones, described in the  next sections, we have to fix our framework. For this reason we denote the space of continuous functions with the flux condition on $\Gamma$ as
$$
\mathcal{C}_\Gamma(\Omega):=\{\Psi\in \mathcal{C}(\Omega_1)\times \mathcal{C}(\Omega_2) : \dfrac{\p \psi_1}{\p {\bf n_1}}= \dfrac{\p \psi_2}{\p {\bf n_1}}= \mu ( \psi_2- \psi_1)\  \hbox{on}\  \Gamma\}.
$$
As for the continuous functions up to the boundary we will write
$$
\mathcal{C}_{\Gamma}(\overline\Omega):=\{\Psi\in \mathcal{C}(\overline\Omega_1)\times \mathcal{C}(\overline\Omega_2) : \dfrac{\p \psi_1}{\p {\bf n_1}}= \dfrac{\p \psi_2}{\p {\bf n_1}}= \mu ( \psi_2- \psi_1)\  \hbox{on}\  \Gamma\}.
$$
 We consider, following~\cite{Pfluger1996}, functions $\psi_i$ defined in $\Omega_i$ belonging to the spaces $\mathcal{C}_{\star}(\Omega_i)$ defined as the set of all  continuous functions with compact support in $\Omega=\Omega_1\cup\Omega_2$; that is $\psi_i$ is the restriction of a function of compact support in $\Omega$ to the subdomain $\Omega_i$ and we define, as before
$$
\mathcal{C}_{\star,\Gamma}(\Omega):=\{\Psi\in \mathcal{C}_\star(\Omega_1)\times \mathcal{C}_\star(\Omega_2) : \dfrac{\p \psi_1}{\p {\bf n_1}}= \dfrac{\p \psi_2}{\p {\bf n_1}}= \mu ( \psi_2- \psi_1)\  \hbox{on}\  \Gamma\}.
$$
Analogously we define the spaces of continuously differentiable and H\"{o}lder continuous functions $\mc{C}_{\Gamma}^{1}$, $\mc{C}_{\Gamma}^{0,\eta}$, $\mc{C}_{\Gamma}^{1,\eta}$, $\mathcal{C}^1_{\star,\Gamma}$, $\dots$
for some $\eta\in (0,1]$.

We will denote by $\mathcal{L}$ the product space $L^2(\Omega_1)\times L^2(\Omega_2)$, and consider its usual scalar product $\langle \Phi,\Psi\rangle_\mathcal{L}=\sum_i\langle \varphi_i,\psi_i\rangle_{L^2(\Omega_i)}$ and {the induced} norm $\|\Phi\|^2_{\mathcal{L}}= \langle \Phi,\Phi\rangle_{\mathcal{L}}=\sum_{i}\|\varphi_i\|_{L^2(\Omega_i)}^2$.
Following~\cite{CiavolellaPaerthame}, let $\mathcal{H}_{0}$ be the Hilbert space of functions $H_0^1(\Omega_1)\times H_0^1(\Omega_2)$, satisfying Dirichlet homogeneous conditions on $\Gamma_1$ and $\Gamma_2$. 
Then, we define $\mathcal{H}_{\Gamma}$ as
$$
 \mathcal{H}_{\Gamma}:=\{\Psi\in \mathcal{H}_0 :
\dfrac{\partial \psi_1}{\partial {\bf n_1}}= \dfrac{\partial \psi_2}{\partial {\bf n_1}}= \mu ( \psi_2- \psi_1)\  \hbox{on}\  \Gamma\},
$$
equipped with the norm
$$
\|\Psi\|^2_{\mathcal{H}_\Gamma}:=\|\nabla \Psi\|^2_{\mathcal{L}}=
\|\nabla\psi_1\|^2_{L^2(\Omega_1)}+\|\nabla\psi_2\|^2_{L^2(\Omega_2)}.
$$
Obviously, $\mathcal{H}_\Gamma$ is a Hilbert space and we define the scalar product in $\mathcal{H}_\Gamma$ as
$$
\left\langle \Phi,\Psi\right\rangle_{\mathcal{H}_\Gamma}:=\left\langle \nabla\varphi_1,\nabla\psi_1\right\rangle_{L^2(\Omega_1)}+\left\langle \nabla\varphi_2,\nabla\psi_2\right\rangle_{L^2(\Omega_2)}.
$$
Thus, since $\mathcal{H}_\Gamma$  is a closed linear subspace of $\mathcal{H}_0$, it is  a Banach space in which the following compact and continuous embedding hold:
\begin{equation}
\label{eq:embedding}
  \mathcal{H}_\Gamma\hookrightarrow \mathcal{L},\quad {\text{and}\quad  \mathcal{H}_\Gamma\rightarrow \mathcal{L}_\Gamma:= L^2(\Gamma)\times L^2(\Gamma)} .
\end{equation}
{Moreover, if $\Phi\in\mathcal{H}_\Gamma$ we have $\int_{\Gamma}  \varphi_i^2=\int_{\partial\Omega_i}  \varphi_i^2 $, and due to the trace inequality,
we find that
\begin{equation*}
 \int_{\Gamma}  \varphi_i^2=\|\varphi_i\|^2_{L^2(\partial\Omega_i)}\leq C_{\Omega_i}\|\varphi_i\|^2_{H^1(\Omega_i)},
\end{equation*}
with $C_{\Omega_i}>0$, depending on the domain $\Omega_i$. It is worth pointing out that, due to Poincare's inequality the norms $\|\cdot \|_{\mathcal{H}_\Gamma}$  and $\|\cdot\|_{\mathcal{H}_0}$ are equivalent. Therefore,  if $C_\Omega$ is a positive constant, we  define a global trace inequality as
\begin{equation*}
\label{eq:global.trace}
\left( \int_{\Gamma}  \varphi_1^2+\int_{\Gamma}  \varphi_2^2\right)=\|\Phi\|^2_{\mathcal{L}_\Gamma}\leq  C_{\Omega}\|\Phi\|^2_{\mathcal{H}_\Gamma}.
\end{equation*}

Finally, we consider the linear operator $\mathbf{L_F}:\mathcal{D}(\mathbf {L_F})=\mathcal{H}_\Gamma\to \mathcal{L}$.
Using this framework,  we can rephrase our original problem~\eqref{main}--\eqref{opbound} as follows:\\  Find $U\in \mathcal{H}_\Gamma$ such that,
\begin{equation}
  \label{eq:main.compact}
(\mathbf{L}_{-\lambda \mathbf{M}}+\mathbf{A}U^{p-1})U=0,
\end{equation}
where $U^p$ stands for $(u_1^p, u^p_2)^T$.

A natural idea of what a positive weak solution is, consists in considering functions that vanish on $\partial\Omega$ but imposing the inside boundary condition.
Hence we define a weak solution to~\eqref{eq:main.compact} by formally multiplying (in $\mathcal{L}$) by a test function $\Phi$ and integrating by parts. The precise definition reads as follows:
\begin{definition}
  Let  $U \in \mathcal{H}_\Gamma $ be  positive. We will say that $U$ is a \emph{(weak) solution}  of~\eqref{eq:main.compact} if $$
   \int_{\Omega_1}\nabla u_1\cdot \nabla \varphi_1+ \int_{\Omega_2}\nabla u_2 \cdot \nabla \varphi_2+\int_\Gamma \mu(u_2-u_1)(\varphi_2-\varphi_1)
 =\lambda \langle \mathbf{M}U,\Phi\rangle_{\mathcal{L}}- \langle \mathbf{A}U^p,\Phi\rangle_{\mathcal{L}},
$$
for $\Phi=(\varphi_1,\varphi_2)^T\in\mathcal{C}_{\star,\Gamma}^\infty$.

We define \emph{ weak sub- and supersolutions} as usual; i.e, by replacing in the
definition of solution equality by $\leq$ or $\geq$ respectively.
\end{definition}

%%%%%%%%%%%%%%%%%%%%%%%%%%%%%%%%%%%%%%%%%%%%

%%%%%%%%%%%%%%%%%%%%%%%%%%%%%%%%%%%%%%%%%%%%

\section{Auxiliary problems}
\label{subsect:eigenvalue_problem}
This section is devoted to state the general properties for two interface linear problems related to~\eqref{main} that we will use in the next sections.
 We collect them  here for the sake of completeness and refer the reader to~\cite{CiavolellaPaerthame}, ~\cite{Suarezetal}, or~\cite{WangSu} for the detailed proofs.

%%%%%%%%%%%%%%%%%%%%%%%%%%%%%%%%%%%%%%%%%%%%
\subsection{A linear problem}
Let us consider the linear problem given by
\begin{equation}
\label{mainlinear}
\left\{ \begin{array}{llll}
(-\Delta+f_1(x)) u_1(x)=g_1(x),&\quad \text{ in } \Omega_1,\quad & 
 \\[6pt]
(-\Delta+f_2(x)) u_2(x)=g_2(x),&\quad \text{ in } \Omega_2,
\quad & 
\end{array}\right. \quad
\end{equation}
together with boundary conditions
\begin{equation}
  \label{mainlinearbdy}
\left\{\begin{array}{ll} \dfrac{\p u_1}{\p {\bf n_1}}= \dfrac{\p u_2}{\p {\bf n_1}}= \mu(u_2- u_1),& \quad \text{on } \Gamma,\\[10pt]
u_i =0,  & \text{\quad on } \Gamma_i, \quad \text{with}\quad  i=1,2,\end{array}\right.
\end{equation}
where $F=(f_1,f_2)^T, G=(g_1,g_2)^T$
and {$F> 0$}.
Using the notation introduced in Section~\ref{sect:preliminaries} we can rewrite this problem in a more compact form as
\begin{equation}
  \label{eq:system.aux}
\mathbf{L_\mathbf{F}}U={G}\quad\hbox{in}\quad \Omega.
\end{equation}
Associated with problem~\eqref{mainlinear}--\eqref{mainlinearbdy} we define the bilinear form $a:\mathcal{H}_\Gamma\to \mathbb{R}$, given by
\begin{equation}
\label{bilinear_form}
a(U,V) =\int_{\Omega_1}\nabla u_1\cdot \nabla v_1+ \int_{\Omega_2}\nabla u_2\cdot \nabla v_2+ \int_{\Omega_1}f_1 u_1 v_1+\int_{\Omega_2}f_2 u_2 v_2+ \int_\Gamma \mu (u_2-u_1)(v_2-v_1),
\end{equation}
and, consequently, a weak solution of~\eqref{mainlinear}--\eqref{mainlinearbdy} is a function  $U\in\mathcal{H}_\Gamma$ such that $$a(U,V)=\langle G,V\rangle_{\mathcal{L}}$$ for all $V\in\mathcal{H}_\Gamma$.

\begin{theorem}
  If $F\in L^\infty(\Omega_1)\times L^{\infty}(\Omega_2)$ and $G\in\mathcal{L}$, then~\eqref{mainlinear}--\eqref{mainlinearbdy} has a unique weak solution $U$. Moreover, there exits a positive constant $C$, such that
  $$
  \|U\|_{\mathcal{H}_\Gamma}\leq C(\|G\|_{\mathcal{L}}+\|U\|_{\mathcal{L}}).
  $$
\end{theorem}

\begin{proof}
The proof might be performed after applying Lax-Milgram Theorem by showing first that the biliniear form $a(\cdot,\cdot)$ is continuous and coercive in the Hilbert space $\mathcal{H}_\Gamma$.
We omit the details here, which can be adapted from~\cite{CiavolellaPaerthame},~\cite{Suarezetal} and~\cite{WangSu}.
\end{proof}

Under these conditions, we define the resolvent operator for~\eqref{eq:system.aux}, $(\mathbf{L}_{\mathbf{F}})^{-1}  \,:\,  \mathcal{L} \to  \mathcal{H}_{\Gamma}$, by
$$
G\in\mathcal{L}\mapsto U:= (\mathbf{L}_{\mathbf{F}})^{-1}  G\in \mathcal{H}_{\Gamma},
$$
having the following.

\begin{lemma}
\label{compacteness}
{The resolvent operator  $(\mathbf{L}_{\mathbf{F}})^{-1} $ is linear, continuous and compact.}
\end{lemma}

\begin{proof}
The continuity and linearity of the resolvent $(\mathbf{L}_{\mathbf{F}})^{-1} $ are clear from Lax-Milgram Theorem. To show the compactness of the resolvent we take a sequence
$G_n\in \mathcal{L}$ such that
$$\lim_{n\to 0} \|G_n-G\|_{\mathcal{L}}=0,$$
and set $U= (\mathbf{L}_{\mathbf{F}})^{-1}  G$, $U_n= (\mathbf{L}_{\mathbf{F}})^{-1}  G_n$ for $n\geq 1$. Now, using the bilinear form~\eqref{bilinear_form}, we get
$$a(U_n-U,U_n-U)=\langle G_n-G,U_n-U\rangle_{\mathcal{L}},\quad \forall n\geq 1.$$
Since $a(\cdot,\cdot)$  is coercive
$$\|U_n-U\|^2_{ \mathcal{H}_{\Gamma}}\leq C|a(U_n-U,U_n-U)|=C |\langle G_n-G,U_n-U\rangle_{\mathcal{L}}|\leq C
\|G_n-G\|_{ \mathcal{L}} \|U_n-U\|_{ \mathcal{H}_{\Gamma}}.
$$
Consequently,
$$\|U_n-U\|_{ \mathcal{H}_{\Gamma}}\leq C \|G_n-G\|_{ \mathcal{L}} \to 0,\quad \hbox{as}\quad n\to \infty,$$
showing that the resolvent is a compact operator.
\end{proof}

Next, by elliptic regularity, see~\cite{LG-book} and~\cite{WangSu}, and also as a direct consequence of the comparison principle stated in~\cite{WangSu}, we have the following result:
\begin{theorem}
 Let $F, G\in  \mathcal{C}^{0,\eta}(\Omega_1)\times \mathcal{C}^{0,\eta}(\Omega_2)$ for some $\eta \in (0, 1)$ and $F> 0$, then~\eqref{eq:system.aux} has a unique solution $U\in \mathcal{C}^{2,\eta}_\Gamma(\Omega)$ that verifies
$$
  \|U\|_{C^{2,\eta}_\Gamma}\leq C\|G\|_{C^{0,\eta}}.
$$
Moreover, if $G\geq 0$,  then $U\geq0$.
 \end{theorem}

To end this subsection, let us consider, for $F\in \mathcal{C}^{0,\eta}(\Omega_1)\times \mathcal{C}^{0,\eta}(\Omega_2)$ the eigenvalue problem
\begin{equation}
  \label{eq:system.aux.eigenvalue}
\mathbf{L_\mathbf{F}}U={\sigma U}\quad\hbox{in}\quad \Omega.
\end{equation}
This problem will be crucial in the sequel to describe the existence of solutions for 
 the nonlinear problem~\eqref{eq:main.compact}. We firs notice that
having a compact and {positive resolvent}, one has that the spectrum of $\mathbf{L_F}$ might contain infinitely many isolated eigenvalues (see for instance \cite{Bre})
and together with the compactness of the operator, we have that the spectrum is discrete and each one of the eigenvalues has finite multiplicity.
 As it is usual in the literature, see for instance~\cite{Am05}, we introduce the following definition.
\begin{definition}
\label{def_princeig}
Given an operator $\mathbf{A}$ and a domain $\mathcal{O}=\mathcal{O}_1\cup\mathcal{O}_2$, we will say that $\Sigma[\mathbf{A} ;\mathcal{O}]$  is the
\emph{principal eigenvalue} of $\mathbf{A}$ in $\mathcal{O}$
if   $\Sigma[\mathbf{A} ;\mathcal{O}]$ is the unique value for which $\mathbf{A}\Phi=\Sigma[\mathbf{A} ;\mathcal{O}]\Phi$ (together with the boundary conditions) 
possesses a solution $\Phi$ {with $\Phi\succeq 0$}. Such a function $\Phi$ is called \emph{principal eigenfunction}.
\end{definition}
{ Note that it is common to {assume} that the principal eigenfunction is positive {in $\mathcal{O}$}. In our context we only impose that one of the two components is positive in $\mathcal{O}_1$ (or $\mathcal{O}_2$), allowing the other one to be identically 0.}

\begin{theorem}
  Problem~\eqref{eq:system.aux.eigenvalue}
  has a principal eigenvalue $\Sigma[\mathbf{L_F};\Omega]$ which is unique and simple (i.e., the algebraic multiplicity is 1).
  Moreover, $\Sigma[\mathbf{L_F};\Omega]$ is continuous and monotone increasing with respect to $\mu$.
\end{theorem}

 The proof of this result is a consequence of the previous analysis and can be found in~\cite{WangSu}}.

\subsection{A weighted linear eigenvalue problem}
Let us now extend the results shown for problem~\eqref{eq:system.aux.eigenvalue} to a weighted eigenvalue problem of the form
\begin{equation}
\label{13}
    \mathbf{L_F} \Phi=
    \lambda \mathbf{M}\Phi\quad\hbox{in}\quad \Omega,\end{equation}
    where $m_i$ are assumed to be two positive $L^\infty(\Omega_i)$ functions.  Problem~\eqref{13} is the short version of the equivalent problem of finding $(\varphi_1, \varphi_2)^T\in\mathcal{H}_\Gamma$ such that
\begin{equation}\label{eq.13.long}
\left\{ \begin{split}
& (-\Delta+f_1) \varphi_1 = \lambda m_1(x) \varphi_1\quad \hbox{in}\quad \Omega_1,\\ &
(-\Delta +f_2)\varphi_2 =  \lambda m_2(x) \varphi_2\quad \hbox{in}\quad \Omega_2.
\end{split}\right.
\end{equation}
{with boundary conditions}
\begin{equation}
\label{eq.13.boundary}
{\dfrac{\partial \varphi_1}{\partial {\bf n_1}} =  \dfrac{\partial \varphi_2}{\partial {\bf n_1}} =\mu (\varphi_2-\varphi_1)\quad \hbox{on}\quad \Gamma, \qquad \varphi_1 = 0\quad \text{on } \Gamma_1,\qquad \varphi_2 = 0 \quad \text{on } \Gamma_2.}
\end{equation}
Similar weighted systems, with the solutions defined in the whole domain $\Omega$, but without the interface boundary condition,  were studied in, for instance, \cite{DelSua}. For the one equation setting see~\cite{Lo3}.

\begin{definition}
\label{def:strongly.positive}
  A function $\Psi\in \mathcal{H}$ is said to be  \emph{strongly positive}, denoted by $\Psi\gg 0$, { if $\Psi\in \mathcal{C}_\Gamma(\Omega)$} and the following two conditions hold:
  \begin{itemize}
    \item[i)]For any $x\in \Omega_i \cup \Gamma$, $\psi_i(x)>0$;\quad and
    \item[ii)] for any $y\in\Gamma_i$,  $\dfrac{\partial \psi_i(y)}{\partial \bf n_i} < 0$ and $\psi_i(y)=0$.
  \end{itemize}
  \end{definition}

\begin{lemma}
\label{Lem23.lin}
Let $\Phi\in\mathcal{H}_\Gamma$ be a non-negative solution of~\eqref{13} associated with a positive eigenvalue $\lambda$. Then $\Phi$ is strongly positive.
\end{lemma}

\begin{proof}
Let $(\varphi_1,\varphi_2)^T\in\mathcal{H}_\Gamma$ be a non-negative solution of  \eqref{eq.13.long}--\eqref{eq.13.boundary}.
Now, following the analysis developed in \cite{MolRossi}, let $v_1\in H_0^1(\Omega_1)$ be a positive solution to the auxiliary problem
\begin{equation}
\label{lin.auxl1}
\left\{\begin{array}{ll} (-\Delta+f_1) v_1=g_1, & \quad \hbox{in}\quad \Omega_1,\\
v_1=0, & \quad \hbox{on}\quad \partial \Omega_1,\end{array}\right. \quad \hbox{with}\quad g_1=\lambda m_1(x) \varphi_1\geq 0.
\end{equation}
By assumption, we have that $\varphi_1\geq 0$ on $\Gamma$
and, hence, $\varphi_1$ is a supersolution for the auxiliary problem \eqref{lin.auxl1}, such that by the comparison principle and the maximum principle
$$\varphi_1\geq v_1>0 \quad \hbox{in}\quad \Omega_1.$$
Furthermore, we show that actually the solution $\varphi_1$ is strictly positive on $\Gamma$. Indeed, if we assume that $\varphi_1(x_0)=0$ for some $x_0\in \Gamma$, due to Hopf's Lemma
we have that
$$   \dfrac{\partial \varphi_1(x_0)}{\partial {\bf n_1}} <0,$$
which means that $x_0$ is not a minimum of $\varphi_1$ in $\overline\Omega_1$, contradicting the fact that $\varphi_1(x_0)=0$. Therefore, $\varphi_1>0$ on $\Gamma$.
Condition (ii) is also a direct consequence of Hopf's Lemma.

On the other hand,
we also consider a second auxiliary problem
\begin{equation}
\label{lin.auxl2}
\left\{\begin{array}{ll} (-\Delta+f_2) v_2= \lambda m_2(x) \varphi_2, & \quad \hbox{in}\quad \Omega_2,\\
v_2=\varphi_1, & \quad \hbox{on}\quad \partial \Omega_2.\end{array}\right.
\end{equation}
Thanks to the strong maximum principle and since $\varphi_1>0$  on $\Gamma \subset \partial \Omega_2$ we find that $v_2>0$ in $\Omega_2$. Consequently, since $\varphi_2$ is
a supersolution to problem \eqref{lin.auxl2} it follows that
$$\varphi_2\geq v_2>0 \quad \hbox{in}\quad \Omega_2,$$
showing the strong positivity of the eigenfunction $\Phi$.

The fact that each $\varphi_i\in \mathcal{C}(\Omega_i)$ follows from elliptic regularity and the previous subsection.
\end{proof}

\begin{remark}
 Having that any non-negative solution is actually a strong positive solution is in some sense a way of having the strong Maximum Principle and the existence of a positive strict supersolution.
In fact, once we have that the first eigenvalue is positive, by the previous Lemma~\ref{Lem23.lin} we find that the eigenfunctions associated with it are positive. This provides us with a
 strict positive supersolution. Indeed, for the operator  $\mathbf{L_F} +\Lambda \I$, with $\Lambda$ a sufficiently large positive constant, positive constants $\Upsilon$ are positive strict supersolutions
 in the sense that
 $$(\mathbf{L}_{\mathbf{F}} +\Lambda \I)\Upsilon>0 \quad \hbox{in}\quad \Omega, \quad\hbox{or}\quad \Upsilon>0 \quad \hbox{on}\quad \partial \Omega.$$
  \end{remark}

Furthermore, we will now prove that the first eigenvalue is actually positive and simple, in the sense of algebraic multiplicity 1, applying Krein-Rutrman Theorem as one of the main ingredients.
To do so, the next definition is an important element in order to apply the classical Krein-Rutman Theorem.

\begin{definition}
  \label{def:strongly.positive.operator}
  Let $P$ be a positive cone with non-empty interior in $\mathcal{H}_\Gamma$ and $T:\mathcal{H}_\Gamma\to \mathcal{H}_\Gamma$ a linear operator. We say that $T$ is strongly positive if
  \begin{equation*}
\label{cone}
T(P\setminus 0)\subset {\rm Int}P.
\end{equation*}
\end{definition}

See~\cite{Am2} and~\cite{Du2006Book} for more details on strongly positive operators and positive cones.  The next theorem shows that~\eqref{13} has indeed a principal eigenvalue.

\begin{theorem}
\label{thm:principal.eigenvalue}
Problem~\eqref{13} has a principal eigenvalue denoted by $\lambda_1$. Moreover, $\lambda_1$ is real, simple (in the sense of
multiplicities) and strictly positive.
\end{theorem}

The proof of this theorem will relay 
on the  analysis of the eigenvalues of two problems related to~\eqref{13}, that we analyse in the following two lemmas.

\begin{lemma}
  \label{lemma:eigenfunctions.K}
  Let $r(\mathcal{K}_\lambda)$ be the spectral radius of the operator $\mathcal{K}_\lambda:\mathcal{H}_\Gamma\to\mathcal{H}_\Gamma$ denoted by
  \begin{equation}
  \label{operatorK_lam}
  \mathcal{K}_\lambda:=(\mathbf{L}_\mathbf{F} +\Lambda \mathbf{I})^{-1} (\lambda \mathbf{M}+(\Lambda+\omega) \I).
  \end{equation}
 For any $\lambda\in\mathbb{R}$ such that
 \begin{equation}
  \label{eq:condition.Lambda}
 \lambda m_i+(\Lambda+\omega) >0.
\end{equation}
 holds, $r(\mathcal{K}_\lambda)$ is
   positive  and simple. Moreover, it is the principal eigenvalue of
   \begin{equation}
  \label{eq:pbm.K}
\mathcal{K}_\lambda\Phi=\sigma \Phi.
\end{equation}
\end{lemma}

\begin{proof} The result is a direct consequence
 of the Krein-Rutman Theorem (cf. \cite[Theorem 3.2]{Am2}),
as well as of several results performed in \cite{Dan1} and \cite{Lo3}.

Let us first
consider, for a fixed value of $\lambda$, the  problem
\begin{equation}
  \label{eq:aux.con.omega}
\mathbf{L}_{\mathbf{F}-\lambda\mathbf{ M}}\Phi = \omega \Phi.
\end{equation}
It is straightforward to see that  $\lambda\in \mathbb{R}$ is an eigenvalue of \eqref{13} if and only if $\omega=0$ is an eigenvalue  of the operator
$\mathbf{L}_{\mathbf{F}-\lambda\mathbf{M}}$.

On the other hand, it  is also easily seen, see~\cite{Lo3} for a weak formulation of the result, that $\omega$ is an  eigenvalue of $\eqref{eq:aux.con.omega}$ if and only if $\sigma=1$ is an eigenvalue of problem \eqref{eq:pbm.K},
where $\mathcal{K}_\lambda$ is an operator denoted by \eqref{operatorK_lam}, for any $\lambda\in\mathbb{R}$. Here $\Lambda$ is a constant (which depends on $\lambda$), chosen sufficiently large, so that
condition \eqref{eq:condition.Lambda} and, thus, $(\mathbf{L}_\mathbf{F} +\Lambda \mathbf{I})^{-1}$
exists.

Then, it turns out that $\mathcal{K}_\lambda$ is a compact, strongly positive operator (see Definition~\ref{def:strongly.positive.operator}). Indeed, first of all observe that thanks to the embedding~\eqref{eq:embedding}, we can ensure that   $\mathcal{K}_\lambda:  \mathcal{H}_\Gamma \rightarrow \mathcal{H}_\Gamma$ is a compact operator.
In order to prove that $\mathcal{K}_\lambda$ is strongly positive, let us take $\Phi_{0}\in P\setminus\{0\}$.
Thus, thanks to the comparison principle we find that the corresponding components are strictly positive in the interior of the subdomains $\Omega_i$ and, thanks to Hopf's Lemma, we also have that the exterior normal derivatives are negative on the boundary of
$\Omega$, where we have homogeneous Dirichlet boundary conditions. In other words, Definition \ref{def:strongly.positive} is satisfied by $\Phi_{0}$, proving that
$\mathcal{K}_\lambda$ is strongly positive.

Consequently, due to Krein-Rutman's Theorem, we have that $\mathcal{K}_\lambda$ has positive spectral radius, $r(\mathcal{K}_\lambda)$, which is a simple eigenvalue of~\eqref{eq:pbm.K}. Moreover, the associated eigenfunction $\Phi\in{\rm Int}P$ is positive, { in the sense that $\Phi\succeq 0$}, and there is no other eigenvalue with a positive eigenfunction. Hence, following Definition~\ref{def_princeig},  $r(\mathcal{K}_\lambda)$ is the principal eigenvalue of~\eqref{eq:pbm.K}.
Indeed, thanks to Lemma\;\ref{Lem23.lin} we actually have that the eigenfunction is strongly positive.
\end{proof}

\begin{remark}
\label{remark:eigenfunctions}
  It can be seen  that, if $\Phi$ is the eigenfunction corresponding to the eigenvalue $\sigma=1$ in problem~\eqref{eq:pbm.K}, then $\Phi$ is
  also an eigenfunction for $\omega$ in problem~\eqref{eq:aux.con.omega}. The same happens for the pair $(0,\Phi)$ and $(\lambda, \Phi)$ corresponding to~\eqref{eq:aux.con.omega} and~\eqref{13}, respectively.
\end{remark}

\begin{lemma}
\label{lemma:spectral.radius.K}
For any $\omega>0$, there exists  $\bar \lambda\in\mathbb{R}$ such that  $r(\mathcal{K}_{\bar\lambda})=1$.

\end{lemma}

\begin{proof}
First of all observe that, since $(\omega+\Lambda)\mathbf{I}+\lambda\mathbf{M}$ is an increasing function of $\lambda$, we have that, as long as~\eqref{eq:condition.Lambda} holds,
$\mathcal{K}_{\lambda_1}<\mathcal{K}_{\lambda_2}$ for $\lambda_1<\lambda_2$. Hence, see~\cite{Am2} and \cite{PJ3},  $r(\mathcal{K}_\lambda)$ is a continuous,  increasing function of $\lambda$.

Let us assume that $\omega>\Sigma[\mathbf{L_F}]$, where $\Sigma[\mathbf{L_F}]$ is the principal eigenvalue of~\eqref{eq:system.aux.eigenvalue}.
Then, for  $\lambda=0$ and using~\eqref{eq:pbm.K}, we get that
\begin{equation}
  \label{eq:spectral.radius.lambda0}r(\mathcal{K}_0)=\frac{\Lambda+\omega}{\Lambda+\Sigma[\mathbf{L_F}]}>1.
\end{equation}
 We claim that there exists $\lambda<0$, with $|\lambda|$ large, such that $r(\mathcal{K}_\lambda)<1$. By continuity and monotonicity of
 the spectral radius this yields the desire conclusion in the case  $\omega>\Sigma[\mathbf{L_F}]$.

To prove such a claim we follow the ideas in~\cite{Lo3}.
Since
$$
\begin{aligned}
(\omega+\Lambda)\mathbf{I}+\lambda\mathbf{M}&=(\omega+\Sigma[\mathbf{L_F}]-1+\Lambda+1-\Sigma[\mathbf{L_F}])\mathbf{I}+\lambda\mathbf{M}\\
&\leq
(\Lambda+\Sigma[\mathbf{L_F}]-1)\mathbf{I}+\underbrace{\big((\omega+1-\Sigma[\mathbf{L_F}])\mathbf{I}+\lambda\mathbf{M}\big)^+}_{:=\mathbf{A}^+},
\end{aligned}
$$where $\mathbf{A}^+$ is defined as the positive part of the matrix, see Section~\ref{S1}, it yields to
$$r(\mathcal{K}_\lambda)<\hat r,$$
being $\hat r$ the spectral radius of the operator $(\mathbf{L_F}+\Lambda\mathbf{I})^{-1}\big((\Lambda+\Sigma[\mathbf{L_F}]-1)\mathbf{I}+\mathbf{A}^+\big)$. Hence, the claim
follows if we show that $\hat r<1$ for $\lambda<0$, with $|\lambda|$ large.

By assumption, the weights $m_i$ are assumed to be two positive and bounded functions. Thus, we have that
$$
\lim_{\lambda\to-\infty}\big((\omega+1-\Sigma[\mathbf{L_F}])\mathbf{I}+\lambda\mathbf{M}\big)^+=\mathbf{0}.
$$
Let us denote by $\overline{r}$ the spectral radius of $(\mathbf{L_F}+\Lambda\mathbf{I})^{-1}\big((\Lambda+\Sigma[\mathbf{L_F}]-1)\mathbf{I}\big)$. It is easy to see that
$$
\overline{r}=\frac{\Lambda+\Sigma[\mathbf{L_F}]-1}{\Lambda+\Sigma[\mathbf{L_F}]}<1.
$$
Finally, we have $\lim_{\lambda\to-\infty} \hat r=  \overline{r}<1$, and we conclude that there exists $\bar \lambda<0$ such that $r(\mathcal{K}_{\bar \lambda})=1.$

Notice that if $0<\omega<\Sigma[\mathbf{L_F}]$, then~\eqref{eq:spectral.radius.lambda0} becomes $r(\mathcal{K}_0)<1$. We can repeat the same
proof by considering the negative part and taking the limit as $\lambda \to\infty$ to get $\bar r>1$. By continuity if $\omega=\Sigma[\mathbf{L_F}]$ then $\bar\lambda=0$.
\end{proof}

\begin{corollary}
\label{lemma:omega}
 For any $\omega>0$, there exists $\bar\lambda\in\mathbb{R}$ such that $\omega$ is the principal eigenvalue of~\eqref{eq:aux.con.omega}.
\end{corollary}

\begin{proof}
Lemma~\ref{lemma:spectral.radius.K} implies that given any $\omega>0$ we can find $\bar \lambda$ such that $\omega$ is an eigenvalue for the operator $\mathbf{L}_{\mathbf{F}-\bar\lambda\mathbf{ M}}$, see~\eqref{eq:aux.con.omega}. The eigenfunctions corresponding to $\omega$ are the same as the eigenfunctions corresponding to $r(\mathcal{K}_{\bar\lambda})$; see Remark~\ref{remark:eigenfunctions}. We conclude the proof by using Lemma~\ref{lemma:eigenfunctions.K}
so that $\omega=\Sigma[\mathbf{L}_{\mathbf{F}-\bar\lambda\mathbf{M}};\Omega]$.
\end{proof}

\begin{proof}[Proof of Theorem~{\rm\ref{thm:principal.eigenvalue}}]
 As we have mentioned above, the problem of analysing the existence of
 a principal eigenvalue of problem~\eqref{13} is equivalent to the problem of finding a zero principal eigenvalue for~\eqref{eq:aux.con.omega}. That is, finding zeros for the function
\begin{equation*}
\label{Sigma_function}
\Sigma(\lambda):=\Sigma[\mathbf{L}_{\mathbf{F}-\lambda\mathbf{M}} ;\Omega].
\end{equation*}
It can be seen that this function is real analytic, continuous,
decreasing and $ \Sigma(\lambda)\to-\infty$ as $\lambda\to\infty$; see~\cite{Am2} and~\cite{Santi} for further details. As a consequence of the previous results, we have that $\Sigma(0)>0$. Thus, there exists $\lambda_1>0$, such that $\Sigma(\lambda_1)=0$, which implies that $\lambda_1$ is an eigenvalue for~\eqref{13}. It is moreover the principal eigenvalue and, as $r(\mathcal{K}_\lambda)$ is simple, so is $\lambda_1$.
\end{proof}

%%%%%%%%%%%%%%%%%%%%%%%%%%%%%%%%%%%%%%%%%%%%%%%%%%%%%%%%%%%%%%%%%%%
\section{Asymptotic behaviour of a spatially heterogeneous linear problem}
\label{Section_asym}
%%%%%%%%%%%%%%%%%%%%%%%%%%%%%%%%%%%%%%%%%%%%%%%%%%%%%%%%%%%%%%%%%%

In this section we ascertain the asymptotic behaviour of a parameter dependent linear elliptic eigenvalue problem that will be crucial in the sequel. In particular, the limiting
problem obtained in this section will provide us with an eigenfunction used in the following sections to characterise the existence of positive solutions by
the method of sub and supersolutions.

Given a real parameter $\alpha$, we consider
the linear weighted elliptic eigenvalue problem
\be
\label{asylin}
\mathbf{L}_{\alpha \mathbf{A}}\Phi_\alpha =\lambda_\alpha {\bf M} \Phi_\alpha
\end{equation}
{or equivalently
\begin{equation}
  \label{eq:asylin.split}
 \left(-\Delta+\alpha a_i(x)\right) \varphi_{i,\alpha}=\lambda_\alpha m_i(x) \varphi_{i,\alpha},\qquad \text{in\ \  }\Omega_i
\end{equation}
 together with the boundary conditions given by~\eqref{opbound}.

 Here, under the
framework explained in the previous section, we assume
that the principal eigenfunction $\Phi_\alpha$ is normalized, so that
\begin{equation}
\label{eq:normalized.phi}
\int_{\Omega_1}m_1(x) \varphi_{1,\alpha}^2+\int_{\Omega_2} m_2(x) \varphi_{2,\alpha}^2=1.
\end{equation}
Moreover, since $\Phi_\alpha\succeq 0$ we can actually conclude that $\lambda_\alpha\geq 0$.   In other words,
$\lambda_\alpha$ is the value such that we have an eigenfunction of the operator
$\mathbf{L}_{\alpha \mathbf{A} -\lambda_\alpha {\bf M}}$ corresponding to  the (principal) eigenvalue zero, i.e.
$$\Sigma[\mathbf{L}_{\alpha \mathbf{A} -\lambda_\alpha {\bf M}};\Omega]=0.$$
We also observe that due to the monotonicity of the principal eigenvalue with respect to the domain (see \cite{Am05} or \cite{Santi})
$$0=\Sigma[\mathbf{L}_{\alpha \mathbf{A} -\lambda_\alpha {\bf M}};\Omega]<\Sigma[\mathbf{L}_{-\lambda_\alpha {\bf M}};\Omega_0],$$
where $\Omega_0=\Omega_0^{a_1}\cup \Omega_0^{a_2}$.
 Furthermore, let us consider the uncoupled problem
$$
\left\{ \begin{array}{l@{\quad}l@{\quad\quad \text{and} \quad\quad }l@{\quad}l}
 -\Delta \varphi_{1,0}=\sigma_1 \varphi_{1,0}   &  \hbox{in}\quad \Omega_0^{a_1},& \varphi_{1,0}=0   &\hbox{on}\quad\partial\Omega_0^{a_1},\\[4pt]
-\Delta \varphi_{2,0}=\sigma_2 \varphi_{2,0}  & \hbox{in}\quad \Omega_0^{a_2},& \varphi_{2,0}=0  &\hbox{on}\quad\partial\Omega_0^{a_2},
\end{array}\right.
$$
where $\sigma_i:=\sigma[-\Delta;\Omega_0^{a_i}]$ denotes the principal eigenvalue for each equation. We define
$$\sigma[\mathbf{L};\Omega_0] =\inf\{\sigma_1, \sigma_2\}, $$ which is simple and positive. Note that the  actual value of $\Sigma[\mathbf{L};\Omega_0]$ depends only on the size of the subdomains $\Omega_0^{a_1}$ and $\Omega_0^{a_2}$. Now, applying the monotonicity of the principal eigenvalue with respect to the potential we find that
$$\Sigma[\mathbf{L}_{-\lambda_\alpha {\bf M}};\Omega_0] \leq \sigma[\mathbf{L};\Omega_0] -\lambda_\alpha m,$$
where $m:=\inf\{m_1(x)|_{\Omega_1}, m_2(x)|_{\Omega_2}\}$.
Therefore, we have the following estimation for the eigenvalues $\lambda_\alpha$,
\begin{equation}
\label{eq:bound.lambda}
\lambda_\alpha<\frac{\sigma[\mathbf{L};\Omega_0]}{m}.
\end{equation}
Thanks to the monotonicity of the principal eigenvalue with respect to the potential we know that the eigenvalues $\lambda_\alpha$ are increasing
in terms of the parameter $\alpha$ and due to \eqref{eq:bound.lambda} bounded above.
Therefore, for a sufficiently big $\alpha$ we can say that the eigenvalues $\lambda_\alpha$ are strictly positive.

The following result will be of great importance in the proof of the main result of this section.

\begin{lemma}
\label{LemCot}
For each fixed $\alpha>0$, let $(\lambda_\alpha,\Phi_\alpha)$ be a solution to~\eqref{asylin}.
Then
\begin{equation*}
\label{cotas}
\begin{split}
	\|\Phi_\alpha\|_{\mathcal{H}_\Gamma}\leq \lambda_\alpha, \quad  \mu \int_{\Gamma} (\varphi_{2,\alpha}- \varphi_{1,\alpha})^2\leq \lambda_\alpha, \quad \alpha \left(\int_{\Omega_1} a_1 \varphi_{1,\alpha}^2 +  \int_{\Omega_2} a_2 \varphi_{2,\alpha}^2\right)\leq \lambda_\alpha.
	\end{split}
\end{equation*}
\end{lemma}

\begin{proof}
Multiplying~\eqref{asylin} by $\Phi_{\alpha}\in\mathcal{H}_\Gamma$
and integrating by parts we get, for the left-hand side of~\eqref{asylin}
\begin{align*}
\langle &\mathbf{L}_{\alpha {\mathbf{A}}} \Phi_\alpha,\Phi_\alpha\rangle_\mathcal{L}\\&=	\int_{\Omega_1} |\nabla \varphi_{1,\alpha}|^2  -\int_{\Gamma} \varphi_{1,\alpha} \frac{\partial \varphi_{1,\alpha}}{\partial \mathbf{n}_1}
	+\int_{\Omega_2}  |\nabla \varphi_{2,\alpha}|^2  +\int_{\Gamma} \varphi_{2,\alpha} \frac{\partial \varphi_{2,\alpha}}{\partial \mathbf{n}_1}   + \alpha \left(\int_{\Omega_1} a_1 \varphi_{1,\alpha}^2 +  \int_{\Omega_2} a_2 \varphi_{2,\alpha}^2\right)\\
&=\int_{\Omega_1} |\nabla \varphi_{1,\alpha}|^2 +\int_{\Omega_2}  |\nabla \varphi_{2,\alpha}|^2
+\int_{\Gamma}  \frac{\partial \varphi_{1,\alpha}}{\partial \mathbf{n}_1}(\varphi_{2,\alpha}-\varphi_{1,\alpha})
	+ \alpha \left(\int_{\Omega_1} a_1 \varphi_{1,\alpha}^2 +  \int_{\Omega_2} a_2 \varphi_{2,\alpha}^2\right).
\end{align*}
Moreover, due to the inside boundary condition, the previous integral over $\Gamma$ can be written as
$$
\int_{\Gamma}  \frac{\partial \varphi_{1,\alpha}}{\partial \mathbf{n}_1}(\varphi_{2,\alpha}-\varphi_{1,\alpha}) =\mu \int_{\Gamma}  (\varphi_{2,\alpha}-\varphi_{1,\alpha})^2.
$$
For the right-hand side of~\eqref{asylin}, we get, since we are considering a normalized $\Phi_\alpha$,
$$
\lambda_\alpha\langle {\bf M}\Phi_\alpha,\Phi_\alpha\rangle_\mathcal{L}=\lambda_\alpha.
$$
Therefore,
\begin{equation*}
\label{varphi_alpha}
\int_{\Omega_1} |\nabla \varphi_{1,\alpha}|^2 +\int_{\Omega_2}  |\nabla \varphi_{2,\alpha}|^2 + \alpha \left(\int_{\Omega_1} a_1 \varphi_{1,\alpha}^2 +  \int_{\Omega_2} a_2 \varphi_{2,\alpha}^2\right)
+\mu\int_{\Gamma}  (\varphi_{2,\alpha}-\varphi_{1,\alpha})^2 =\lambda_\alpha,
\end{equation*}
and we conclude the result.
\end{proof}

Subsequently we want to analyse the asymptotic behaviour of~\eqref{asylin}, when the parameter $\a$ goes to infinity. To do so, we consider also  the following limit uncoupled problem
\begin{equation}
\label{limaspo}
\left\{ \begin{array}{l@{\quad}l@{\quad\quad \text{and} \quad\quad }l@{\quad}l}
 -\Delta \varphi_{1,\infty}=\lambda_\infty m_1(x)\varphi_{1,\infty}   &  \hbox{in}\quad \Omega_0^{a_1},& \varphi_{1,\infty}=0   &\hbox{on}\quad\partial\Omega_0^{a_1},\\[4pt]
-\Delta \varphi_{2,\infty}= \lambda_\infty m_2(x) \varphi_{2,\infty}  & \hbox{in}\quad \Omega_0^{a_2},& \varphi_{2,\infty}=0  &\hbox{on}\quad\partial\Omega_0^{a_2},
\end{array}\right.
\end{equation}
where
\begin{equation}
\label{limaspo2}\lambda_\infty:=\inf\{\lambda^{m_1}[-\Delta,\Omega_0^{a_1}],\lambda^{m_2}[-\Delta,\Omega_0^{a_2}]\},
\end{equation}
stands for the principal eigenvalue of problem \eqref{limaspo} associated with the normalized  eigenfunction
$\Phi_\infty\in H_0^1(\Omega_0^{a_1})\times H_0^1(\Omega_0^{a_2})$, with
\begin{equation}
\label{zero_+}\varphi_{1,\infty} =0\quad \hbox{in} \quad \Omega\setminus \Omega_0^{a_1}\quad \hbox{and}\quad
\varphi_{2,\infty} =0\quad \hbox{in} \quad \Omega\setminus \Omega_0^{a_2},\end{equation} and
$\lambda^{m_{i}}[-\Delta,\Omega_0^{a_i}]$ are the  principal eigenvalues of  $-\Delta \varphi_{i}=\lambda m_{i}(x)\varphi_{i}$
in the domain $\Omega_0^{a_i}$ under homogeneous Dirichlet boundary conditions.
{Recall that, thanks to the definition of eigenfunction, we allow that one of the components of $\Phi_\infty$ is equal to zero, so that~\eqref{limaspo} may have a trivial equation}.

\begin{theorem}
\label{limprob}
Let $\Omega_0^{a_1}$ and $\Omega_0^{a_2}$ be  two non-empty subdomains and let $(\lambda_\alpha,\Phi_\alpha)$ and $(\lambda_\infty,\Phi_\infty)$ be a solution to~\eqref{asylin} and~\eqref{limaspo} respectively.
Then, Problem \eqref{asylin}
converges to the  limiting problem~\eqref{limaspo} when the parameter $\alpha$ goes to infinity in the sense that $\lambda_\alpha\to\lambda_\infty$ and
$\Phi_\alpha\to \Phi_\infty$ in $\mathcal{H}_{\Gamma}$.
\end{theorem}

Before we state the proof of this theorem, let us discuss, through different possible situations, the limiting behaviour of the eigenvalue $\lambda_\infty$ depending on the geometrical
configuration of the vanishing subdomains $\Omega_0^{a_1}$ and $\Omega_0^{a_2}$.

Assume first that
\begin{equation}
\label{eq:case_2}
\lambda_\infty=\lambda^{m_1}[-\Delta,\Omega_0^{a_1}]=\lambda^{m_2}[-\Delta,\Omega_0^{a_2}].
\end{equation}
Then, each equation~\eqref{limaspo} is satisfied with positive uncoupled eigenfunctions that concentrate in the corresponding subdomain of
{``more than enough resources"}, $\Omega_0^{a_i}$:
\begin{equation*}
  \varphi_{1,\infty}> 0 \quad \hbox{in}\quad \Omega_0^{a_1}\quad \hbox{and}\quad \varphi_{2,\infty} > 0 \quad \hbox{in}\quad
  \Omega_{0}^{a_2}.
\end{equation*}
On the other hand, if we assume that, for example,
\begin{equation}
\label{eq:case_1}
\lambda_\infty=\lambda^{m_1}[-\Delta,\Omega_0^{a_1}]<\lambda^{m_2}[-\Delta,\Omega_0^{a_2}],
\end{equation}
(which might happen, for instance, if $\Omega_0^{a_2}$ is bigger than $\Omega_0^{a_1}$ and $m_1=m_2$ at every single point), then the limit eigenfunction concentrates in $\Omega_0^{a_1}$ being zero in the rest of $\Omega$, i.e. $\Phi_\infty$ is defined by
$$\varphi_{1,\infty}>0 \quad \hbox{in}\quad \Omega_0^{a_1} \quad \hbox{and}\quad \varphi_{2,\infty}=0 \quad \hbox{in}\quad \Omega_2.$$
In general, we can conclude that for limit function
\begin{equation}
\label{eq:case_1_gen}
\Phi_\infty=\left\{\begin{array}{ll} (\varphi_{1,\infty},0)^T & \hbox{if}\quad \lambda^{m_1}[-\Delta,\Omega_0^{a_1}]<\lambda^{m_2}[-\Delta,\Omega_0^{a_2}],\\
(0,\varphi_{2,\infty})^T & \hbox{if}\quad \lambda^{m_2}[-\Delta,\Omega_0^{a_2}]<\lambda^{m_1}[-\Delta,\Omega_0^{a_1}].\end{array}\right.
\end{equation}
Observe that the limit eigenfunction pair is just given in the corresponding domain $\Omega_i$ while the limit eigenfunction $\Phi_\infty$ will be zero on both components in $\Omega\setminus\Omega_i$
respectively.

\begin{proof}[Proof of Theorem~{\rm\ref{limprob}}]
Let $\{\alpha_n\}_{n\geq 1}$ be an increasing unbounded sequence.
For each $\alpha_n$ we take the
normalised  principal eigenfunctions $\Phi_{\alpha_n}$ of~\eqref{asylin} with principal eigenvalue $\lambda_{\alpha_n}$.
Note that then, thanks to Lemma\;\ref{LemCot} and since the coefficients $m_1$ and $m_2$ are two  positive $L^\infty$-functions, the norms for the eigenfunctions $\Phi_{\alpha_n}$ are bounded in $\mathcal{H}_\Gamma$. Moreover,
since the embedding $\mathcal{H}_\Gamma\hookrightarrow \mathcal{L}$
is compact, there exist a subsequence of $\{\alpha_n\}_{n\geq 1}$,
again labeled by $n$, and $\Phi_\infty\in\mathcal{L}$ such that
\begin{equation*}
 \lim_{n\to\infty} \|\Phi_{\alpha_n}-\Phi_\infty\|_{\mathcal{L}}=0,
\end{equation*}
strongly and weakly in $\mathcal{H}_\Gamma$.
Thus, we can extract a subsequence, again
labelled by $\{\Phi_{\alpha_n}\}$, weakly convergent
in $\mathcal{H}_\Gamma$ and strongly in $\mathcal{L}$
to some function $\Phi_\infty$.

Next, we will prove that $\{\Phi_{\alpha_n}\}_{n\geq 1}$ is  actually a
Cauchy sequence in $\mathcal{H}_\Gamma$. This implies that $\Phi_\infty\in \mathcal{H}_\Gamma$
and
\begin{equation}
\label{b45}
    \lim_{n\to \infty} \|\Phi_{\alpha_n}-\Phi_\infty\|_{\mathcal{H}_\Gamma}=0.
\end{equation}
{Indeed, let $n<m$ so that $0<\alpha_n<\alpha_m$. We define
$$D_{n,m}:=\|\nabla(\Phi_{\alpha_n}-\Phi_{\alpha_m})\|_\mathcal{L}=\underbrace{\int_{\Omega_1} \left|\nabla (\varphi_{1,\alpha_n} - \varphi_{1,\alpha_m})\right|^2}_{D_{n,m}^1}
  +  \underbrace{\int_{\Omega_2} \left|\nabla (\varphi_{2,\alpha_n} - \varphi_{2,\alpha_m})\right|^2}_{D_{n,m}^2}.$$
Using separately each of the equations in system \eqref{eq:asylin.split} and taking into account  the boundary conditions \eqref{opbound} it gives that
\begin{align*}
      D_{n,m}^1    =&
      \int_{\Omega_1} \left|\nabla\varphi_{1,\alpha_n}\right|^2+
    \int_{\Omega_1}\left|\nabla\varphi_{1,\alpha_m}\right|^2- 2
    \int_{\Omega_1}\left\langle \nabla\varphi_{1,\alpha_n},\nabla\varphi_{1,\alpha_m}\right\rangle\\
         =&\int_{\Omega_1} ( \lambda_{\alpha_n} m_1\varphi_{1,\alpha_n}-
    \alpha_n a_1 \varphi_{1,\alpha_n}) \varphi_{1,\alpha_n} +\mu \int_{\Gamma} (\varphi_{2,\alpha_n}-\varphi_{1,\alpha_n}) \varphi_{1,\alpha_n}
    \\  &+ \int_{\Omega_1} ( \lambda_{\alpha_m} m_1\varphi_{1,\alpha_m}
    -\alpha_m a_1 \varphi_{1,\alpha_m}) \varphi_{1,\alpha_m} +\mu \int_{\Gamma} (\varphi_{2,\alpha_m}-\varphi_{1,\alpha_m}) \varphi_{1,\alpha_m}
    \\ & -2 \int_{\Omega_1} ( \lambda_{\alpha_n} m_1\varphi_{1,\alpha_n}
    -\alpha_n a_1 \varphi_{1,\alpha_n}) \varphi_{1,\alpha_m} -2\mu \int_{\Gamma} (\varphi_{2,\alpha_n}-\varphi_{1,\alpha_n}) \varphi_{1,\alpha_m}.
\end{align*} 
The term $D_{n,m}^2$ is similar.
Thus, rearranging terms, we are driven to
\begin{align*}
      D_{n,m}  &  =  \lambda_{\alpha_n}\int_{\Omega_1}  m_1\varphi_{1,\alpha_n}(\varphi_{1,\alpha_n}-\varphi_{1,\alpha_m})
    +\lambda_{\alpha_n}\int_{\Omega_2} m_2\varphi_{2,\alpha_n} (\varphi_{2,\alpha_n}-\varphi_{2,\alpha_m})
    \\ & \hspace{0.5cm} + \lambda_{\alpha_m}\int_{\Omega_1} m_1\varphi_{1,\alpha_m}
    (\varphi_{1,\alpha_m}-\varphi_{1,\alpha_n})
     + \lambda_{\alpha_m}\int_{\Omega_2} m_2\varphi_{2,\alpha_m}(\varphi_{2,\alpha_m}-\varphi_{2,\alpha_n})
      \\ & \hspace{0.5cm}  +  (\lambda_{\alpha_m} - \lambda_{\alpha_n})\int_{\Omega_1} m_1\varphi_{1,\alpha_n} \varphi_{1,\alpha_m}
      + (\lambda_{\alpha_m} -   \lambda_{\alpha_n}) \int_{\Omega_2} m_2\varphi_{2,\alpha_n} \varphi_{2,\alpha_m}+R_{n,m}+F_{n,m},
\end{align*}
where we have denoted by $F_{n,m}$ the sum of all the terms involving integrals over $\Gamma$ and $R_{n,m}$ are the remaining terms (terms involving integrals in $\Omega_i$ and the potentials $a_i$).
Since $\{\alpha_n\}$ is increasing we have
$$
  R_{n,m}  =\sum_{i=1,2}\left(
    - \alpha_n \int_{\Omega_i} a_i (\varphi_{i,\alpha_n}-\varphi_{i,\alpha_m})^2-(\alpha_m-\alpha_n)\int_{\Omega_i} a_i \varphi_{i,\alpha_m}^2\right)
     \leq 0.
$$
On the other hand,
\begin{align*}
  F_{n,m}  &= -\mu \int_{\Gamma} (\varphi_{2,\alpha_n}- \varphi_{1,\alpha_n})^2-\mu \int_{\Gamma}(\varphi_{2,\alpha_m}- \varphi_{1,\alpha_m})^2 +2\mu\int_\Gamma (\varphi_{2,\alpha_n}- \varphi_{1,\alpha_n}) (\varphi_{2,\alpha_m}- \varphi_{1,\alpha_m})\\
  &=-\mu\int_\Gamma \left((\varphi_{2,\alpha_n}- \varphi_{1,\alpha_n}) - (\varphi_{2,\alpha_m}- \varphi_{1,\alpha_m})\right)^2\leq 0.
  \end{align*}
In analysing the terms where the eigenvalues $\lambda_{\alpha_n}$ are involved we take into consideration that the sequence $\{\lambda_{\alpha_n}\}$ is increasing,
thanks to the monotonicity of the principal eigenvalue with respect to the potential and bounded above due to the estimation \eqref{eq:bound.lambda}. Hence and
subsequently, applying H\"{o}lder's inequality, {the order of the sequence $\{\lambda_{\alpha_n}\}$ and the upper bound for the eigenvalues $\lambda_{\alpha_n}$}
$${D_{n,m} \leq C \left[\int_{\Omega_1} \left( \varphi_{1,\alpha_n} - \varphi_{1,\alpha_m}\right)^2
+  \int_{\Omega_2} \left(\varphi_{2,\alpha_n} - \varphi_{2,\alpha_m}\right)^2\right],\quad \hbox{with $C$ a positive constant}.}$$
Therefore, \eqref{b45} is satisfied and we have that the limit function $\Phi_\infty\geq 0$ is normalised in the sense of~\eqref{eq:normalized.phi}.
}

The next step consists in showing that $\Phi_\infty$ is indeed a solution to~\eqref{limaspo}.
First, we claim that~\eqref{zero_+} holds; that is, $\varphi_{i,\infty}=0$ outside $\Omega_0^{a_i}$.
In fact, thanks to Lemma~\ref{LemCot} it follows that
\begin{equation*}
   \lim_{n\to \infty} \left(\int_{\Omega_1}  a_1\varphi^2_{1,\alpha_n}+
   \int_{\Omega_2} a_2 \varphi^2_{2,\alpha_n}\right)=0.
\end{equation*}
Moreover, since the sets $\Omega_i$ are disjoint, we have that each of the integrals above is equal to zero and hence, due to the normalization of $\Phi_\infty$ and
after applying H\"{o}lder's inequality, we find that
that
{\begin{align*}
  \left| \int_{\Omega_i} a_i \varphi_{i,\alpha_n}^2 - \int_{\Omega_i} a_i \varphi_{i,\infty}^2 \right| &
  \leq \int_{\Omega_i} a_i(\varphi_{i,\alpha_n}+\varphi_{i,\infty})|\varphi_{i,\alpha_n}-\varphi_{i,\infty} | \\ & \leq
  \max_{\bar \Omega_i} a_i \cdot \|\varphi_{i,\alpha_n}\!-\!\varphi_{i,\infty}\|_{L^2(\O_i)}\left(
  \int_{\Omega_i} \varphi_{i,\alpha_n}^2\! +\! \varphi_{i,\infty}^2\! +\! 2  \varphi_{i,\alpha_n}\varphi_{i,\infty}
  \right)^{1/2} \\ & \leq 2 C \max_{\bar \Omega_i} a_i
  \cdot \|\varphi_{i,\alpha_n}-\varphi_{i,\infty}\|_{L^2(\Omega_i)}\to 0,
\end{align*}
where $C$ is a positive constant which depends on the coefficients $m_i$.} Therefore,
\begin{equation*}
 \int_{\Omega_1} a_1\varphi^2_{1,\infty}+
   \int_{\Omega_2} a_2 \varphi^2_{2,\infty}=0,
\end{equation*}
which concludes the proof of \eqref{zero_+}, since $a_i$ are nonnegative and identically zero in $\Omega_0^{a_i}$.

Consequently, we show that
\eqref{zero_+} implies that
\begin{equation*}
\label{313}
  \Phi_\infty \in H_0^1(\Omega_{0}^{a_1})\times
  H_0^1(\Omega_{0}^{a_2}).
\end{equation*}
Indeed for a sufficiently small $\delta>0$,
consider the open sets
\begin{equation}
\label{eq:set.omega.delta}
  \Omega_{\delta}^{a_i}:=\{\,x\in\Omega_i\;:\;\mathrm{dist}(x,\Omega_{0}^{a_i})<\delta\,\}.
\end{equation}
According to \eqref{zero_+},
\begin{equation*}
  \Phi_\infty \in H_0^1(\Omega_{\delta}^{a_1})\times
  H_0^1(\Omega_{\delta}^{a_2}),
\end{equation*}
and, hence, there exists $\delta_0>0$ such that
\begin{equation*}
\Phi_\infty \in \bigcap_{0<\delta<\delta_0}
 \left(H_0^1(\Omega_{\delta}^{a_1})\times
  H_0^1(\Omega_{\delta}^{a_2})\right).
\end{equation*}
On the other hand, since $\Omega_{0}^{a_1}$ and $\Omega_{0}^{a_2}$ are smooth
subdomains of $\Omega$, they are stable in the sense of Babuska and
V\'yborn\'y \cite{BV} and,
therefore,
\begin{equation*}
  H_0^1(\Omega_{0}^{a_1})\times
  H_0^1(\Omega_{0}^{a_2})= \bigcap_{0<\delta<\delta_0}
  \left(H_0^1(\Omega_{\delta}^{a_1})\times
  H_0^1(\Omega_{\delta}^{a_2})\right).
\end{equation*}
Finally, we pass to the limit in the weak formulation of \eqref{asylin} to show that $\Phi_\infty$ is indeed a (weak) solution of~\eqref{limaspo}. To this aim, consider the test function
\begin{equation*}
  \Upsilon=(\upsilon_1,\upsilon_2)^T \in\mathcal{C}_0^{\infty}(\Omega_{0}^{a_1})
  \times \mathcal{C}_0^{\infty}(\Omega_{0}^{a_2}).
\end{equation*}
Multiplying \eqref{asylin} (assuming $\alpha=\alpha_n$, with $n>1$) by $\Upsilon$ and
integrating by parts  gives rise to
$$
    \int_{\Omega_{0}^{a_i}} \nabla \upsilon_i \cdot \nabla \varphi_{i,\alpha_n}
        = \lambda_{\alpha_n} \int_{\Omega_{0}^{a_i}} \upsilon_i m_i \varphi_{i,\alpha_n}.
$$
Consequently, passing to the limit as $n\to \infty$, it is
apparent that $\Phi_\infty$ provides us with a
weak solution of the uncoupled problem
\begin{equation*}
\label{limit:alpha_problem}
-\Delta \varphi_{1,\infty}=\ell m_1(x) \varphi_{1,\infty}  \quad   \hbox{in}\quad \Omega_0^{a_1},\qquad\hbox{and}\qquad
  -\Delta \varphi_{2,\infty}= \ell m_2(x)\varphi_{2,\infty} \quad \hbox{in}\quad \Omega_0^{a_2},
\end{equation*}
together with~\eqref{zero_+} and
$
\varphi_{i,\infty}\geq 0$ in $\Omega_0^{a_i}$ and where
$\ell=\lim_{n\to \infty} \lambda_{\alpha_n}$.
Moreover, by elliptic regularity  $\Phi_\infty$ is indeed a classical solution $ \mathcal{C}^{2,\eta}(\overline\Omega_{0}^{a_1})
  \times \mathcal{C}^{2,\eta}(\overline\Omega_{0}^{a_2})$.
Therefore, $\ell=\lambda_\infty.$
\end{proof}

%%%%%%%%%%%%%%%%%%%%%%%%%%%%%%%%%%%%%%%%%%%%%
%%%%%%%%%%%%%%%%%%%%%%%%%%%%%%%%%%%%%%%%%%%%%

\section{Existence of solutions for the Non-degenerate case}

\label{sec:nondegenerate}
%%%%%%%%%%%%%%%%%%%%%%%%%%%%%%%%%%%%%%%%%%%%%
%%%%%%%%%%%%%%%%%%%%%%%%%%%%%%%%%%%%%%%%%%%%%

\noindent In this section we characterise the
existence and uniqueness of positive solutions $U$ of problem~\eqref{eq:main.compact}  in terms of the parameter $\lambda$ in the case in which
the coefficients of the non-linear terms $a_i$ are strictly positive at every point of the domain $\Omega_i$.
To this aim we define $\lambda^*$ as
the principal eigenvalue of the problem
\begin{equation}
\label{eq:eigenvalue_lambda}
\mathbf{L}\Phi= \lambda^* \mathbf{M}\Phi  \qquad \text{ in } \Omega,
\end{equation}
under the boundary conditions~\eqref{opbound}, whose existence was analysed in Section\;\ref{subsect:eigenvalue_problem}. Note that $\lambda^*$ can also be understood
  as the value such that the principal eigenvalue $\Sigma\left[\mathbf{L}_{-\lambda \mathbf{M}};\Omega\right]$ is $0$.

\begin{theorem}
\label{Theouni}
Let 
$\Omega_0^{a_1}=\Omega_0^{a_2}=\emptyset$. For any $\lambda$, problem~\eqref{eq:main.compact}
admits a unique positive solution $U\in {\mathcal C}^{2,\eta}_\Gamma(\overline\Omega)$
if and only if
\be
\label{lambin}
\lambda >\lambda^*.
\ee
\end{theorem}

\begin{remark}
\label{rem.lambda*}
  Observe that condition~\eqref{lambin} is equivalent to $\Sigma\left[\mathbf{L}_{-\lambda \mathbf{M}};\Omega\right]<0$, since the function
  $\Sigma(\lambda)=\Sigma\left[\mathbf{L}_{-\lambda \mathbf{M}};\Omega\right]$ is continuous and decreasing with respect to $\lambda$, see~\cite{PJ3} for similar problems.
  \end{remark}

\begin{proof}
The proof, that uses the method of sub and supersolutions, see \cite{Am1}, \cite{Am2} for further details, follows similar arguments to those shown in \cite[Theorem 3.7]{AC-LG2} so we omit the details.
\end{proof}

In the next result we prove that there exists a branch of positive solutions of~\eqref{eq:main.compact} emanating from the trivial solution $(\lambda,{ U})=(\lambda,{ 0})$
at $\lambda=\lambda^*$. This, indeed, means that $\lambda^*$ is a bifurcation point to
a smooth curve of solutions of~\eqref{eq:main.compact},~\cite{CR}. Moreover, since there is no other bifurcation point, this branch of positive solutions goes to infinity,~\cite{R1}.
Additionally, we will show that the solutions are monotone with respect to the parameter $\lambda$.

\begin{theorem}
\label{Th2.0} For any $\lambda >\lambda^*$, let $\theta(\lambda):= U_\lambda$ be the unique positive solution of problem~\eqref{eq:main.compact}. Then, the map
\begin{equation*}
\label{maplam}
  \theta: \left(\lambda^*,\infty\right) \  \longrightarrow \
  \mathcal{C}_\Gamma (\Omega),\text{\quad such that\quad}
  \l \  \mapsto \
   \theta(\lambda)
\end{equation*}
is of class $\mathcal{C}^1$.
Moreover, $\theta(\lambda)> \theta(\eta)$ if $\lambda>\eta
> \lambda^*$, and $\theta$ bifurcates from $(\lambda,{ U})=(\lambda,{ 0})$ at
$\lambda= \lambda^*$, i.e.,
\begin{equation}
\label{410}
  \lim_{\lambda \downarrow\lambda^*} \theta(\lambda)={ 0}.
\end{equation}
\end{theorem}

\begin{proof}
Let us consider the operator  $\mf{F}: \mathbb{R} \times \mathcal{C}_\Gamma (\overline \Omega) \longrightarrow
 \mathcal{C}_\Gamma (\overline \Omega),$
defined by
$$
  \mf{F}(\lambda,{ U}):= (\mathbf{I}-  \mathbf{L}^{-1}{(\lambda \mathbf{M}-\mathbf{A}{ U}^{p-1})}){ U},$$
where $\mathbf{L}^{-1}$ denotes the \lq\lq inverse laplacian matrix\rq\rq, denoted by~\eqref{inverse.lap}. We know that $\mf{F}$ is of class
$\mathcal{C}^1$ and, by elliptic regularity, $\mf{F}(\lambda,\cdot)$
is a compact perturbation of the identity for every $\lambda\in\mathbb{R}$.

We observe that, for each fixed $\lambda>\lambda^*$, if $U_\lambda\in\mathcal{C}_\Gamma(\overline\Omega)$ is the positive solution to~\eqref{eq:main.compact}, then $\mf{F}(\lambda,U_\lambda)=~0$.
Moreover,  since any non-negative solution turns out to be strongly positive, we have
\begin{equation}
\label{413}
  0 = \Sigma[\mathbf{L}_{-\lambda \mathbf{M}+\mathbf{A}{ U}_\lambda^{p-1}};\Omega].
\end{equation}
Differentiating  $\mf{F}$ with respect to ${ U}$, we have that,
for every ${ U}\in \mc{C}_{\Gamma}(\overline \Omega)$,
\begin{equation*}
  D_{{ U}} \mf{F}(\lambda,{ U}_\lambda){U}  = \big({\bf I}- \mathbf{L}^{-1}{(\lambda \mathbf{M}-p\mathbf{A}{ U_\lambda}^{p-1})} \big){ U},
\end{equation*}
and, in particular, as consequence of the compactness of the inverse operator $\mathbf{L}^{-1}{\mathbf{M}}$, we get that $D_{{ U}} \mf{F}(\lambda,{ U}_\lambda)$ is a Fredh\"{o}lm
operator of index zero, since  it is a compact perturbation of the
identity map, see~\cite{Bre}. Moreover, we claim that it is injective, and hence  a linear topological isomorphism. Indeed, let ${ U}\in \mc{C}_{\Gamma}(\overline \Omega)$ such
that
$$\big(\mathbf{I}-\mathbf{L}^{-1}{(\lambda\mathbf{M}-p\mathbf{A}U_\lambda^{p-1})}\big)U=0.
$$
Then, by elliptic regularity, ${ U}\in \mc{C}_{\Gamma}^{2,\eta}(\overline \Omega)$
and, we have that
\begin{equation}
  \label{eq:inyective}
  \mathbf{L}_{-\lambda\mathbf{M}+p\mathbf{A}U_\lambda^{p-1}}U=0.
\end{equation}
On the other hand, owing to the monotonicity of the principal
eigenvalue with respect to the potential,~\cite{Santi}, and since $p>1$, we find from \eqref{413} that
\begin{equation}
\label{eq:monotonicity.potential}
  \Sigma[\mathbf{L}_{-\lambda \mathbf{M}+p\mathbf{A}{ U}_\lambda^{p-1}};\Omega]>0
\end{equation}
which implies, together with~\eqref{eq:inyective} that  ${ U}={0}$, and hence, that $D_{{ U}} \mf{F}(\lambda,{ U}_\lambda)$ is injective. Therefore, we have that $D_{{U}} \mf{F}(\lambda,{U}_\lambda)$ is a linear topological isomorphism and hence, it is invertible.

Moreover, if we differentiate the nonlinear operator $\mf{F}(\lambda,{ U_\lambda})$ with respect to the parameter $\lambda$ we have that
$$
D_{{U}} \mf{F}(\lambda,{U}_\lambda)\frac{dU_\lambda}{d\lambda}=-D_\lambda\mf{F}(\lambda,{ U_\lambda}).
$$
Applying the operator $\mathbf L$ on both sides, this latter expression yields
$$\mathbf{L}_{-\lambda \mathbf{M}+p\mathbf{A}{U_\lambda}^{p-1}} \frac{d {U}_\lambda}{d \lambda}=\mathbf{M}{U}_\lambda.$$
Therefore, since~\eqref{eq:monotonicity.potential} holds, thanks to the characterisation of the maximum principle
\cite{LM} we find that $\theta'(\lambda)=\frac{d {U}_\lambda}{d \lambda}$ is positive
and, then,  ${ U}_\lambda$ is increasing with respect to $\lambda$. Moreover, due to the uniqueness of the positive solutions and the application of the Implicit Function Theorem it follows that the map $\theta(\lambda)$ is
of class $\mathcal{C}^1$.

Finally, to analyse the bifurcation of $\theta(\lambda)$ we observe that  $\mf{F}(\l,{0})={0}$ for all $\lambda\in\mathbb{R}$ and
\begin{equation*}
  D_{{ U}} \mf{F}(\lambda,{0}){ U} = (\mathbf{I}-  \lambda \mathbf{L}^{-1}{\mathbf{M}}){ U},\qquad \lambda\in\mathbb{R},
  \;\; { U} \in  \mathcal{C}_\Gamma (\overline \Omega).
\end{equation*}
For each $\lambda\in\mathbb{R}$, we denote by $\mf{L}_\lambda$ the linear operator $D_{{U}} \mf{F}(\lambda,{0})$. Also,
$\mf{L}_\lambda$ is real analytic in $\lambda$, since it is a compact perturbation of
the identity of linear type with respect to $\lambda$. Thus, as a consequence, there exists a $\lambda_0$ such that $\mf{L}_{\lambda_0}U=0$
if and only if there is $U\neq 0$, $U\in\mathcal{C}_\Gamma (\overline \Omega)$ such that
\begin{equation}
\label{crank}
\mathbf{L}U=\lambda_0\mathbf{M}U  \quad   \hbox{in}\  \Omega \qquad {\rm and}\qquad
U=0 \quad \hbox{on}\   \partial\Omega.
\end{equation}
Thus, by definition of $\lambda^*$  we get $\lambda_0:= \lambda^*$ and
associated with it there is a unique solution $\Phi_*\succ
0$ of \eqref{crank}, up to a multiplicative constant. It is clear, then, that ${\rm Ker}[\mf{L}_{\lambda^*}]= \mathrm{span }(\Phi_*)$. Now, set
$$ \mf{L}_1:= \frac{d \mf{L_\lambda}}{d \lambda}=-\mathbf{L}^{-1}{\mathbf{M}}.$$
Moreover, the
following  condition holds:
\begin{equation}
\label{412}
  \mf{L}_1 \Phi_* \not\in {\rm R}[  \mf{L}_{\lambda^*} ],
\end{equation}
see~\cite{CR}. Indeed,
suppose by contradiction that  there exists $U\in \mc{C}_{\Gamma}(\overline \Omega)$ such that
\begin{equation*}
(\mathbf{I}-  \lambda^* \mathbf{L}^{-1}{\mathbf{M}}){ U} = -\mathbf{L}^{-1}{\mathbf{M}}\Phi_*.
\end{equation*}
Thanks to elliptic regularity we have $U \in \mc{C}_{\Gamma}^{2,\eta}(\overline \Omega)$ and
$
\mathbf{L}_{-\lambda^*\mathbf{M}}U=\mathbf{M}\Phi_*.
$
Multiplying by $\Phi_*$ and integrating by parts in $\Omega$ it follows that
\begin{align*}
\int_{\Omega_1} u_1 (-\Delta-\lambda_* m_1(x))\varphi_{1,*}-\int_{\partial\Omega_1} u_1 \frac{\p \varphi_{1,*}}{\p{\bf n}_1} & +
\int_{\Omega_2} u_2 (-\Delta-\lambda_*m_2(x))\varphi_{2,*}-\int_{\partial\Omega_2} u_2 \frac{\p \varphi_{2,*}}{\p{\bf n}_2}\\ & =-\left(\int_{\Omega_1} m_1(x)\varphi_{1,*}^2+\int_{\Omega_2} m_2(x)\varphi_{2,*}^2\right).
\end{align*}
The left hand side will be zero since it represents the weak expression for the linear eigenvalue problem \eqref{crank}
with a test function $U\in  \mc{C}_{\Gamma}^{2,\eta}(\overline \Omega)$. However, this is impossible, because the right hand side is negative.
Therefore, condition \eqref{412} holds. Consequently, according
to the main theorem of Crandall and Rabinowitz \cite{CR},
$(\lambda,{U})=(\lambda^*,{0})$ is a bifurcation point from
$(\lambda,{ U})=(\lambda,{0})$ to a smooth curve of positive solutions of
\eqref{eq:main.compact}, since $\Phi_*\succ 0$. Moreover, due to the uniqueness proved in
Theorem\,\ref{Theouni},
condition \eqref{410} holds. Finally, applying the global bifurcation theorem of Rabinowitz \cite{R1} such a smooth curve of positive solutions is actually an unbounded branch of
positive solutions since there is only one simple eigenvalue for the problem~\eqref{crank}.
\end{proof}

%%%%%%%%%%%%%%%%%%%%%%%%%%%%%%%%%%%%%%%%%%%%%%

%%%%%%%%%%%%%%%%%%%%%%%%%%%%%%%%%%%%%%%%%%%%%%

\section{Existence of solutions for the degenerate case}
\label{sec:degenerate}
%%%%%%%%%%%%%%%%%%%%%%%%%%%%%%%%%%%%%%%%%%%%%%

\noindent We are now concerned with the case in which $a_1=0$ and $a_2=0$ in some open subdomains of $\Omega_1$ and $\Omega_2$; that is, we assume spatial heterogeneities such that
$\Omega_0^{a_1}\neq \emptyset$ and  $\Omega_0^{a_2}\neq \emptyset$.

\begin{definition}
We say that $\underline{U}$ is a nonnegative subsolution (respectively $\overline{U}$ is a nonnegative supersolution)
to equation \eqref{eq:main.compact} if $U\geq 0$ and
$$
\left\{\begin{array}{ll}
\mathbf L_{-\lambda\mathbf{M}-\mathbf{A}\underline U^{p-1}}\underline U\leq 0 \quad ({\rm resp.\ } \geq)& \quad \hbox{in}\quad \Omega,\\[5pt]
\underline{U}\leq 0  \quad ({\rm resp.\ } \geq), & \quad \hbox{on}\quad \partial \Omega, \\[5pt]
\dfrac{\partial \underline{u}_1}{\partial {\bf n_1}}\leq \mu( \underline{u}_2- \underline{u}_1) \quad  \dfrac{-\partial \underline{u}_2}{\partial {\bf n_1}} \leq
\mu ( \underline{u}_2- \underline{u}_1) & \quad \hbox{on}\quad \Gamma,  \quad ({\rm resp.\ } \geq).\\
\end{array}\right.
$$
\end{definition}

\begin{theorem}
\label{Theouni2}
Let $\Omega_0^{a_1}\neq\emptyset, \Omega_0^{a_2}\neq\emptyset$. Then, Problem~\eqref{eq:main.compact}
admits a unique  positive solution $U\in {\mathcal C}_\Gamma^{2,\eta}(\overline\Omega)$ if and only if
\begin{equation}
\label{lambin2}
0<\lambda^*<\lambda<\lambda_\infty,
\end{equation}
where $\lambda^*$ is given by~\eqref{eq:eigenvalue_lambda} and $\lambda_\infty$ is given by~\eqref{limaspo2}.
\end{theorem}

\begin{remark}
\label{rem.lambda*2}
  { Analogously to Remark~\ref{rem.lambda*}, condition~\eqref{lambin2} is equivalent to
  \begin{equation}
    \label{eq:equivalent.condition.lambda}
    \Sigma\left[\mathbf{L}_{-\lambda \mathbf{M}};\Omega\right]<0<\Sigma\left[\mathbf{L}_{-\lambda \mathbf{M}};\Omega_0\right].
  \end{equation}
  Indeed,  the function $\Sigma(\lambda):= \Sigma\left[\mathbf{L}_{-\lambda \mathbf{M}};\mathcal{O}\right]$, for any domain $\mathcal{O}$, is continuous and decreasing in $\lambda$, so that there exists a unique value for the parameter $\lambda$, say $\hat\lambda$, for which $\Sigma(\lambda)=0$. Thus, if $\Sigma(\lambda)$ stands for the principal eigenvalue of the problem $\mathbf{L}_{-\lambda \mathbf{M}} \Phi= \Sigma(\lambda)\Phi$ in $\mathcal{O}$,
  we will find that $\Sigma(\lambda)>0$ if $\lambda<\hat\lambda$ and  $\Sigma(\lambda)<0$ if $\lambda>\hat\lambda$.
Using now, $\Omega$ and $\Omega_0$ instead of $\mathcal{O}$ we arrive at
  condition~\eqref{lambin2}. Recall that for $\Omega_0$ we have characterised $\hat\lambda$ as $\lambda_\infty$ in Section~\ref{Section_asym}.

  }
\end{remark}

\begin{proof}
Let us assume that ${U}\in {\mathcal C}_\Gamma^{2,\eta}(\overline\Omega)$ is a positive solution of problem \eqref{eq:main.compact}. Thanks to the uniqueness of the principal eigenvalue we have
$$ \Sigma\left[\mathbf{L}_{-\lambda \mathbf{M}+\mathbf{A}{U}^{p-1}};\Omega\right]=0.$$
Then, applying the monotonicity of the principal eigenvalue with respect to the potential, $$ \Sigma\left[\mathbf{L}_{-\lambda \mathbf{M}};\Omega\right]<\Sigma\left[\mathbf{L}_{-\lambda \mathbf{M}+\mathbf{A}{ U}^{p-1}};\Omega\right]=0.$$
Moreover, due to the monotonicity of the principal eigenvalue with respect to the domain and the spatial configuration of $a_i$ it follows that
\begin{equation*}
\label{eq.sigma.cero}
0=\Sigma\left[\mathbf{L}_{-\lambda \mathbf{M}+\mathbf{A}{\bf U}^{p-1}};\Omega\right]<\Sigma\left[\mathbf{L}_{-\lambda \mathbf{M}};\Omega_0\right]
.\end{equation*}
Note that,  depending on the order of the eigenvalues given by \eqref{eq:case_2} or \eqref{eq:case_1}, the value of $\lambda_\infty$ might be different. However, we
are not distinguishing those cases here. It is just the smallest one.

On the other hand, if~\eqref{lambin2} holds we obtain the existence of positive solutions for problem~\eqref{eq:main.compact} applying the method of sub and supersolutions. %\deleted{, see \cite{Am1} for further details.}

First we choose the supersolution and to do so,
let us consider for sufficiently small $\delta>0$, the sets $\Omega_\delta^{a_i}$ defined in~\eqref{eq:set.omega.delta}. By the continuous dependence
with respect to the domains, see~\cite{Am05},
\begin{equation*}
  \lim_{\delta\to 0} \Sigma[-\Delta -\lambda m_i;\Omega_{\delta}^{a_i}]=\Sigma[-\Delta -\lambda m_i;\Omega_0^{a_i}].
\end{equation*} Therefore, by assumption, for sufficiently small $\delta>0$,
\begin{equation}
\label{55}
 0< \Sigma[-\Delta -\lambda m_i;\Omega_{\delta}^{a_i}] <
  \Sigma[-\Delta -\lambda m_i;\Omega_0^{a_i}],
\end{equation}
see~\eqref{eq:equivalent.condition.lambda}.
Let
$\varphi_{\delta,i}$,  denote the
principal eigenfunction associated with $\Sigma[-\Delta -\lambda m_i;\Omega_{\delta}^{a_i}]$, for $\delta$ fixed
 with zero Dirichlet data on $\partial\Omega_\delta^{a_i}$.
Now, consider $\Psi=(\psi_1,\psi_2)^T$ defined as
\begin{equation*}
  \psi_i:= \left\{ \begin{array}{ll}
  \varphi_{\delta,i} & \hbox{in}\;\;
  \overline\Omega_{\delta/2}^{a_i},\\
  \varphi_{+,i} & \hbox{in}\;\;\Omega_i\setminus\overline \Omega_{\delta/2}^{a_i},\end{array}\right.
\end{equation*}
 where $\varphi_{+,i}$ is any smooth extension, positive
and separated away from zero, chosen such that $\Psi\in \mathcal{C}^2_\Gamma(\Omega)$.
Note that $\varphi_{+,i}$ exists since
$\varphi_{\delta,i}$ is positive and bounded away from
zero on $\partial \Omega_{\delta/2}^{a_i}$.

Subsequently, we show that $k\Psi$
is a supersolution of~\eqref{eq:main.compact} for sufficiently large
$k$. Indeed, by construction, $k\Psi$ verifies the boundary condition on $\Gamma$ and it is non-negative on $\partial\Omega$.
Moreover, since $a_i=0$ in $\Omega_{\delta/2}^{a_i}$ and $k>0$, from~\eqref{55} we have
$$
  (-\Delta -\lambda m_i )k\varphi_{\delta,i}= \Sigma[-\Delta -\lambda m_i;\Omega_{\delta}^{a_i}]
    k\varphi_{\delta,i} \geq
 0\quad \text{\  in\ }\Omega_{\delta/2}^{a_i}.
$$
Finally, in $\Omega_i\setminus \overline\Omega_{\delta/2}^{a_i}$ it follows that
\begin{equation*}
 (-\Delta - \lambda m_i) \varphi_{+,i}  + a_i k^{p-1} \varphi_{+,i}^{p}\geq 0,
\end{equation*}
for sufficiently large $k>1$, since  $a_i$ and $\varphi_{+,i}$ are positive
and bounded away from zero.
Therefore, $k \Psi$ provides us with a
supersolution of \eqref{eq:main.compact} for sufficiently large $k >1$.

As a subsolution we take $\epsilon \Phi_0$, for $0<\epsilon\ll 1$, and where $\Phi_0$
is the eigenfunction associated with the principal eigenvalue {$\Sigma[\mathbf{L}_{-\lambda\mathbf{M}};\Omega]$}. Indeed
$$
\epsilon (-\Delta -\lambda m_i) \varphi_{i,0}  =\epsilon \Sigma\left[\mathbf{L}_{-\lambda \mathbf{M}};\Omega\right] \varphi_{i,0} <-a_i(x) \epsilon^p \varphi_{i,0}^p
\quad \hbox{in}\quad \Omega_i,
$$
{for sufficiently small  $\epsilon>0$,} since $\Sigma\left[\mathbf{L}_{-\lambda \mathbf{M}};\Omega\right]< 0$, see~\eqref{eq:equivalent.condition.lambda}. On the boundary we have the equality.

Once we have a subsolution and a supersolution to problem~\eqref{eq:main.compact}, applying standard iteration arguments
we are provided with the existence of a positive solution for the system~\eqref{eq:main.compact}.
Moreover, as for the non-degenerate case, the
uniqueness follows again as shown in \cite[Theorem 3.7]{AC-LG2}.
\end{proof}

\begin{remark}
\label{rem:increasing.for.a.0}
  Theorem~\ref{Th2.0} holds in the case we are considering in this section since the monotonicity with respect to the potential used in~\eqref{eq:monotonicity.potential} only requires $a_i$ to be different from zero in sets of positive measure;
  see \cite{PaAn} for a discussion on that matter.
\end{remark}

Next we analyse the asymptotic behaviour of  the solution ${ U}_\lambda$ when the parameter $\lambda$
is in the interval $(\lambda^*,\lambda_\infty)$ but approximates to $\lambda_\infty$, both inside the sets $\Omega_0^{a_1}$ and $\Omega_0^{a_2}$ and outside them.

 We have already seen in Theorem\;\ref{Th2.0} that ${ U}_\lambda$
is strictly increasing for $\l\in (\l^*,\l_\infty)$, see also Remark~\ref{rem:increasing.for.a.0}. However, we shall prove that positive solutions actually blow-up when $\lambda$ approaches $\lambda_\infty$ in the respectively vanishing
domains $\Omega_0^{a_1}$ and $\Omega_0^{a_2}$ and depending on  \eqref{eq:case_2} or \eqref{eq:case_1}.

\begin{theorem}
\label{Leminft}
For any fixed $\lambda\in (\lambda^*,\lambda_\infty)$ let ${U}_\lambda=(u_{1,\lambda}, u_{2,\lambda})^T$ be the unique positive solution of~\eqref{eq:main.compact}.  Then, as $\lambda\to \lambda_\infty$:
\begin{itemize}
\item If $\lambda_\infty=\lambda^{m_1}[-\Delta,\Omega_0^{a_1}]<\lambda^{m_2}[-\Delta,\Omega_0^{a_2}]$ then $u_{1,\lambda}$
tends to infinity uniformly on every compact subset of $\Omega_0^{a_1}$, while $u_{2,\lambda}\equiv 0$.
\item If $\lambda_\infty=\lambda^{m_2}[-\Delta,\Omega_0^{a_2}]<\lambda^{m_1}[-\Delta,\Omega_0^{a_1}]$ then
$u_{2,\lambda}$
tends to infinity uniformly on every compact subset of $\Omega_0^{a_2}$ while $u_{1,\lambda}\equiv 0$.
\item If $\lambda_\infty=\lambda^{m_1}[-\Delta,\Omega_0^{a_1}]=\lambda^{m_2}[-\Delta,\Omega_0^{a_2}]$ then
$u_{1,\lambda}$ and $u_{2,\lambda}$
tend to infinity uniformly on every compact subset of $\Omega_0^{a_1}$ and $\Omega_0^{a_2}$, respectively.
\end{itemize}
\end{theorem}
\begin{proof}
Let us consider a sequence $\{\lambda_n\}$ converging from below to $\lambda_\infty$ as $n\to\infty$ and the corresponding unique solutions $U_{\lambda_n}$, which we denote, for simplicity, $U_n$. Take two compact subsets $K_i\subset \Omega_0^{a_i}$. Now, for a fixed $\lambda_n$, let $\Phi_n$ be the principal eigenfunction associated with the
principal eigenvalue $\lambda_{n}$ of~\eqref{asylin}. Furthermore, thanks to the convergence of the linear problem~\eqref{asylin} shown in Section\;\ref{Section_asym},  it follows that
$$\Phi_{n}\to \Phi_\infty\quad \hbox{in}\quad H^1(K_1)\times  H^1(K_2),$$
as $n\to \infty$, where $\Phi_\infty$ is a solution of~\eqref{limaspo} and $\varphi_{1,\infty}$ or $\varphi_{2,\infty}$ might be identically $0$, depending on~\eqref{eq:case_2},~\eqref{eq:case_1} or~\eqref{eq:case_1_gen}. Hence, without loss of generality, let us assume that $\lambda_\infty=\lambda^{m_1}[-\Delta,\Omega_0^{a_1}]<\lambda^{m_2}[-\Delta,\Omega_0^{a_2}]$. The other two cases are handled analogously.

It is straightforward to see that, if $\underline U=\alpha_n^{\frac{1}{p-1}} \Phi_{n}$, then $\underline u_1$ satisfies
$$
-\Delta \underline u_1-\lambda_n m_1 \underline u_1-a_1 \underline u_1^p\leq 0\ \ \text{for\ }x\in\Omega_1, \qquad \underline u_1|_{x\in\Gamma_1}=0,\qquad {\frac{\partial \underline u_1}{\partial{\bf n_1}}-\mu(\underline u_2-\underline u_1)|_{x\in\Gamma}=0}.
$$
Moreover, if $C$ is a big enough constant, then $C \varphi_2$ is a supersolution in $K_2$ to
$$
-\Delta v-\lambda_n m_2 v=0,\quad   v=\sup_{x\in\partial K_2} u_{2,n}(x).
$$
Since $\varphi_{1,\infty} >0$ in $K_1$ and $\varphi_{2,\infty} =0$ in $K_2$, by comparison we find that
$$u_{1,n}\geq \alpha_n^{\frac{1}{p-1}} \varphi_{1,n} \rightarrow \infty\  \hbox{uniformly in}\ K_1,\quad \text{and}\quad u_{2,n}\leq C\varphi_{2,n}\to 0\  \hbox{uniformly in}\  K_2.$$
Moreover, due to the convergence of the eigenfunctions proved in Section~\ref{Section_asym} it follows that $u_{2,n} \to 0$ uniformly in $\overline{\Omega}_2$.
To conclude the proof we must prove convergence up to the boundary of $\Omega_0^{a_1}$ for the component $u_{1,n}$. The proof consists on a geometric construction based upon an argument shown in~\cite{DuHuang} and argues by contradiction.

Since $\lambda_n>0$ for all $n\geq 1$ we have that $-\Delta u_{1,n}=\lambda_n m_1 u_{1,n}\geq 0$ in $\Omega_0^{a_1}$.
Then, due to the maximum principle it is enough to prove
\begin{equation*}
\label{limfro}
u_{1,n}(x_n):=\min_{\partial \Omega_0^{a_1}} u_{1,n}(x)\rightarrow \infty,\quad \hbox{as}\quad n\to \infty,
\end{equation*}
where $x_n\equiv x_{\lambda_n}\in \partial\Omega_0^{a_1}$. Assume, by contradiction that there is a subsequence such that
\begin{equation}
\label{bound_inft}
u_{1,n}(x_n)\leq \kappa\quad \hbox{for all $n\geq 1$ and $x_n\in \partial\Omega_0^{a_1}$,}
\end{equation}
with $\kappa$ a positive constant.

Due to the smoothness of $\partial\Omega_0^{a_1}$, there exists $R>0$ and a map $y\,:\,\partial\Omega_0^{a_1} \rightarrow \Omega_0^{a_1}$,
such that for every $x\in \partial\Omega_0^{a_1}$
\begin{equation}
\label{con_y_ball}
B_R(y(x))\subset \overline\Omega_0^{a_1},\quad \overline B_R(y(x))\cap \partial\Omega_0^{a_1}=\{x\}.
\end{equation}
Indeed, the map $y$ provides us with the centre of the balls in $\Omega_0^{a_1}$ satisfying \eqref{con_y_ball}.
Observe that $\partial\Omega_0^{a_1}\subset \Omega_1$, and the boundaries do not touch, $\partial \Omega_1\cap \partial\Omega_0^{a_1}=\emptyset$.  In particular,
$$u_{1,n}(x) \geq u_{1,n}(x_n)\quad \hbox{for each}\quad x\in \overline B_R(y(x_n)).$$
Define $A_R(y(x_n))=B_R(y(x_n))\setminus \overline B_{R/2}(y(x_n))$ and consider the problem
\begin{equation}
\label{aux56}
\left\{ \begin{array}{ll}
 -\Delta u=\lambda_n m_1 u  &  \hbox{in}\quad A_R(y(x_n)),\\
u= u_{1,n}(x_n)+c_{n}\left(e^{-\delta R^2/4}-e^{-\delta R^2}\right) & \hbox{on}\quad \partial B_{R/2}(y(x_n)),\\
u=u_{1,n}(x_n)& \hbox{on}\quad \partial B_{R}(y(x_n)),
\end{array}\right.
\end{equation} where
\begin{equation*}
\label{cn_inf}
c_{n}= \frac{\min_{\overline  B_{R/2}(y(x_n))}u_{1,n}(x) - u_{1,n}(x_n)}{e^{-\delta R^2/4}-e^{-\delta R^2}}.
\end{equation*}
It is clear that $u_{1,n}(x_n)+c_{n}\left(e^{-\delta R^2/4}-e^{-\delta R^2}\right)\leq u_{1,n}(x)$, for all $x\in \overline   B_{R/2}(y(x_n))$. Hence, $u_{1,n}(x)$ is a supersolution of the problem~\eqref{aux56}.

Similarly, if we define for $\delta>0$ and $x\in  A_R(y(x_n))$ the barrier function of exponential type (as in the proof of the
Hopf-Oleinik boundary lemma)
$$w_{n}(x):=e^{-\delta |x-y(x_n)|^2}-e^{-\delta R^2},$$ we can see that $u_n(x_n)+c_n w_n$ is a subsolution of \eqref{aux56}. Indeed,
a simple computation gives
$$\left(-\Delta - \lambda_n m_1\right)w_n(x)=\left(2\delta N-4\delta^2|x-y(x_n)|^2-\lambda_n m_1(x_n)\right) e^{-\delta |x-y(x_n)|^2}+\lambda m_1 e^{-\delta R^2}.$$
Thus, for $\eta>0$, there exists $\delta>0$ large enough such that
\begin{equation*}
\label{eq:omega_inf}
\left(-\Delta - \lambda_n m_1\right)w_n(x) \leq -\eta<0 \quad \hbox{in}\quad A_R(y(x_n)).
\end{equation*}
Therefore, due to the comparison principle
for~\eqref{aux56}, we have
\begin{equation}
  \label{eq:comp.anillo}
  u_{1,n}(x) >u_{1,n}(x_n)+c_{n} w_{n}\quad \hbox{for every}\quad x\in \overline A_R(y(x_n)). \
\end{equation}

Finally, choose a compact set $K\subset\subset\Omega_0^{a_i}$ such that
$\cup_{n=1}^\infty B_{R/2}(y(x_n)) \subset K$. Since  $u_{1,n}(x)\to\infty$ uniformly in $K$ and, by assumption,  $u_{i,n}(x_n)<\kappa$ for all $n$, we have that \begin{equation}
\label{cinf}
c_{n}\rightarrow \infty,\quad \hbox{as}\quad n\to \infty\quad \hbox{(in other words, when $\lambda_n$ goes to $\lambda_\infty$)}.
\end{equation}
Furthermore, setting the normalised direction
$${\bf n}_{n}:=\frac{y(x_n)-x_n}{R},$$
it follows, using~\eqref{eq:comp.anillo}, that the partial derivative in that direction yields
\begin{align*}
\frac{\partial u_{1,n}}{\partial {\bf n}_{n}}(x_n) & =\lim_{t\to 0}\frac{u_{1,n}(x_n+t {\bf n}_{n})-u_{1,n}(x_n)}{t} \geq c_{n} \lim_{t\to 0} \frac{w_{n}(x_n+t {\bf n}_{n})}{t}\\ &
= c_{n} \lim_{t\to 0} \frac{e^{-\delta |x_n+t {\bf n}_{n}-y(x_n)|^2}-e^{-\delta R^2}}{t}
\\ & = c_{n} \lim_{t\to 0} \frac{e^{-\delta |t {\bf n}_{n}-{\bf n}_{n} R|^2}-e^{-\delta R^2}}{t}
\geq c_{n} \lim_{t\to 0} \frac{e^{-\delta (t-R)^2}-e^{-\delta R^2}}{t}=c_{n} 2\delta R e^{-\delta R^2}.
\end{align*}
Consequently, due to \eqref{cinf},
\begin{equation}
\label{liminf89}
\lim_{n\to \infty}\frac{\p u_{1,n}}{\p {\bf n}_n}(x_n)=\infty.
\end{equation}
On the other hand, we claim that
\begin{equation}
  \label{eq:claim.derivative.bound}
  \frac{\partial u_{1,n}}{\partial {\bf n}_{n}}(x_n)\leq  C,
\end{equation}
contradicting \eqref{liminf89} and proving that $u_{1,n}\to\infty$ uniformly in $\overline\Omega_0^{a_1}$.

We deal now with the proof of the claim~\eqref{eq:claim.derivative.bound}. To this aim, we  consider the auxiliary problem
\begin{equation}
\label{aux78}
-\Delta u=\lambda_n m_1(x) u -a_1(x)u^p \quad    \hbox{in}\quad \Omega_1\setminus \overline\Omega_0^{a_1},\\
\end{equation}
with boundary conditions
\begin{equation}
\label{aux78.bd}
u= u_{1,n}(x_n)\quad   \hbox{on}\quad  \partial \Omega_0^{a_1},\qquad
u= 0 \quad   \hbox{on}\quad  \Gamma_1,\qquad
\dfrac{\partial u}{\partial {\bf n_1}}+\mu u= C\quad \hbox{on}\quad   \Gamma.
\end{equation}
where $C>0$ depends on $u_{2,n}$. Problem~\eqref{aux78}--\eqref{aux78.bd} admits a unique positive solution, $v_n$, for every $\lambda_n\leq \lambda_\infty$ sufficiently close, see~\cite{Santi2}. Thus, if $C$ is chosen so that
$\min(u_{2,n})$ on $\Gamma$ is bigger than $C$, then  $u_{1,n}$ is a positive strict supersolution for~\eqref{aux78}--\eqref{aux78.bd} and, by comparison
\begin{equation}
  \label{eq:aux.comparison.omega} v_n(x) \leq u_{1,n}(x),\quad \hbox{for any}\quad  x\in \overline \Omega_1\setminus \Omega_0^{a_1}.
\end{equation}
Moreover, let $v_\infty$  be the unique positive solution of
$$
-\Delta u=\lambda_\infty m_1(x) u -a_1(x)u^p  \quad   \hbox{in}\quad \Omega_1\setminus \overline\Omega_0^{a_1}$$
and
$$
u= \kappa\quad \hbox{on}\quad \partial \Omega_0^{a_1},\qquad
u= 0 \quad\hbox{on}\quad \Gamma_1,\qquad
\dfrac{\partial u}{\partial {\bf n_1}}+\mu u= 0 \quad \hbox{on}\quad \Gamma,
$$
which is, again by comparison, a supersolution to~\eqref{aux78}--\eqref{aux78.bd}, with $\kappa$ as the upper bound of $u_n$ in~\eqref{bound_inft}.
By comparison we have $v_{n}<v_{\infty}$
and, in particular,
$\|v_{n}\|_{L^\infty(\overline \Omega_1\setminus\Omega_0^{a_1})}$ has a bound independent of $n$. Due to the $L^p$ estimates and the Sobolev embedding theorem we have that $\{v_{n}\}$ is a bounded
sequence in $\mathcal{C}^{1,\eta}(\overline \Omega_1\setminus \Omega_0^{a_1})$ and, hence,
$\|\nabla v_{n}\|_{L^\infty(\overline \Omega_1\setminus\Omega_0^{a_1})}\leq C$,
for a positive constant $C$. Finally, due to~\eqref{eq:aux.comparison.omega} and sincce $v_{n}(x_n)=u_{1,n}(x_n),$
we conclude that
$$\frac{\partial u_{1,n}}{\partial {\bf n}_{n}}(x_n)\leq \frac{\partial v_{n}}{\partial {\bf n}_{n}}(x_n)\leq C,$$
which proofs the claim~\eqref{eq:claim.derivative.bound} and contradicts~\eqref{liminf89}.
\end{proof}

Next, we analyse what happens outside the sets $\Omega_0^{a_1}$ and $\Omega_0^{a_2}$. We deal first with the case $\lambda_\infty=\lambda^{m_1}[-\Delta,\Omega_0^{a_1}]=\lambda^{m_2}[-\Delta,\Omega_0^{a_2}]$.  To this aim we consider the problem
\be
\label{minL}
-\Delta u_i=\lambda_\infty m_i(x) u_i-a_i(x) u_i^p \quad   \hbox{in}\quad \Omega_i\setminus\overline\Omega_0^{a_i},
\ee
with boundary conditions
\be
\label{minL_bdy}
\dfrac{\partial u_{1}}{\partial {\bf n_1}}=  \dfrac{\partial u_{2}}{\partial {\bf n_1}} =\mu (u_{2}-u_{1})\quad   \hbox{on}\  \Gamma,\qquad u_i=0\quad  \hbox{on}\ \Gamma_i
\ee
and
\be
\label{minL_bdy2}
u_i=\infty\quad  \hbox{on}\ \partial \Omega_0^{a_i}.
\ee
As is common in the literature, by $u_i=\infty$ on $\partial\Omega_0^{a_i}$ we mean that $
 u_i(x) \to\infty$ {as} ${\rm dist}(x,\partial\Omega_0^{a_i})
 \to 0.$

\begin{lemma}
  \label{lemma:minimal.solution.aux}
  For any $\lambda\in(-\infty,\infty)$, Problem~\eqref{minL}--\eqref{minL_bdy2} has a minimal positive solution.
\end{lemma}

\begin{proof}
  The proof of this result follows the same argument as in~\cite{DuHuang}.
\end{proof}

\begin{theorem}
\label{lamlinTh}
For any fixed $\lambda\in (\lambda^*,\lambda_\infty)$ let ${U}_\lambda$ be the unique positive solution of~\eqref{eq:main.compact}.
If $\lambda_\infty=\lambda^{m_1}[-\Delta,\Omega_0^{a_1}]=\lambda^{m_2}[-\Delta,\Omega_0^{a_2}]$, then
$U_{\lambda} \to {U}_{\lambda_\infty}$ uniformly on compact subsets of $\overline\Omega\setminus(\overline{\Omega}_0^{a_1}\cup \overline{\Omega}_0^{a_2})$,
as $\lambda\to \lambda_\infty$, where $U_{\lambda_\infty}$ is the minimal positive solution of~\eqref{minL}--\eqref{minL_bdy2}.
\end{theorem}
\begin{proof}
Consider an increasing sequence $\{\lambda_n\}$
which converges to $\lambda_\infty$ as $n\to \infty$ and let  ${ U}_n:=U_{\lambda_n}$ be the corresponding unique positive solutions to~\eqref{eq:main.compact}.

First, we first show that the sequence $\{{U}_n\}$ is uniformly bounded on every compact subset of $\overline\Omega \setminus(\overline{\Omega}_0^{a_1}\cup \overline{\Omega}_0^{a_2})$. 
For a sufficiently small $\epsilon>0$ consider the $\epsilon$--neighbourhoods
$$\Omega_{i,\epsilon}:=\{x\in \Omega_i \setminus \overline{\Omega}_0^{a_i}\,;\,{\rm dist}(x,\overline{\Omega}_0^{a_i})>\epsilon\},$$
which are smooth and non-empty.

For a positive constant $c_\e$, such that $a_i(x)>c_\epsilon$ in $\overline{\Omega}_{i,\epsilon}$
we consider, for a fixed $n$, the problem
\be
\label{minL23}
\left\{ \begin{array}{ll}
 -\Delta u_i=\lambda_\infty m_i (x)u_i-c_\epsilon u_i^p  &   \hbox{in}\quad \Omega_{i,\epsilon},\\
\dfrac{\partial u_{1}}{\partial {\bf n_1}}=  \dfrac{\partial u_{2}}{\partial {\bf n_1}} =\mu( u_{2}-u_{1}) &  \hbox{on}\quad \Gamma,\\
u_i=0 & \hbox{on}\quad \Gamma_i.
\end{array}\right.
\ee
together with the boundary condition
\be
\label{minL23_bdy}
u_i=u_{i,n} \quad  \hbox{on}\quad \partial\Omega_{i,\epsilon}.\\
\ee

Thanks to Theorem~\ref{Theouni2}, Problem \eqref{minL23}--\eqref{minL23_bdy} possesses a classical solution which is actually unique. Moreover, by construction, ${ U}_n$ is a subsolution for problem \eqref{minL23}--\eqref{minL23_bdy}.

Next, we look for a supersolution to~\eqref{minL23}--\eqref{minL23_bdy}. To do so, consider $Z$, the  unique solution  to the auxiliary problem~\eqref{minL23}, see~\cite{DuHuang}, and replace the boundary condition~\eqref{minL23_bdy} by
$$
z_i=\infty \quad \hbox{on}\quad\partial\Omega_{i,\epsilon}.
$$
It is straightforward to see that $Z$ is a subsolution to~\eqref{minL23}--\eqref{minL23_bdy}.
Therefore,  due to the comparison principle
$$z_i\geq u_{i,n} \quad \hbox{in} \quad \Omega_{i,\epsilon}.$$
Moreover, since $Z$ is bounded in $\overline{\Omega}_{i,2\epsilon}$, see~\cite{DuHuang}, we have that there exists a positive constant such that  $u_{i,n}\leq C$ in $\overline{\Omega}_{i,2\epsilon}$, for every $n\geq 1$.
This implies, since $\epsilon>0$ is arbitrary, that $U_n$ is in fact uniformly bounded on compact sets of $\overline\Omega \setminus(\overline{\Omega}_0^{a_1}\cup \overline{\Omega}_0^{a_2})$.

Now, since $U_n$ are uniformly bounded and monotone (see Theorem~\ref{Theouni} and Remark~\ref{rem.lambda*2}) we have that $U_n$ converges to a limit function $U_{\lambda_\infty}$ in $\overline\Omega \setminus(\overline{\Omega}_0^{a_1}\cup \overline{\Omega}_0^{a_2})$. Furthermore, by regularity we can pass to the limit in problem~\eqref{main} and~\eqref{opbound} to get that $U_{\lambda_\infty}$ actually verifies~\eqref{minL}--\eqref{minL_bdy}. So, we only have to verify that the limit function $U_{\lambda_\infty}$ verifies~\eqref{minL_bdy2}.

Indeed, since $U_n$ increases to $U_{\lambda_\infty}$ as $n\to \infty$, we have that $U_{\lambda_\infty}>U_{k}$, for any $k\geq 1$. Now suppose that
$$\lim_{{\rm dist}(x,\partial\Omega_0^{a_i})\to 0} u_{i,\lambda_\infty} =\infty,\quad \hbox{uniformly for}\quad x\in \overline{\Omega}_i\setminus \Omega_0^{a_i},$$
is not true. Then, there exists a sequence $x_n\in \overline{\Omega}_i\setminus \Omega_0^{a_i}$ such that $u_{i,\lambda_\infty}(x_n)\leq C$ for any $n\geq 1$ and some constant $C>0$. So that we have
$u_{i,k}(x_n)\leq C$ for all $n\geq 1$ and $k\geq 1$. On the other hand, owing to Theorem~\ref{Leminft} we also know that $u_{i,k}(x_\infty)\to \infty$ as $k\to \infty$, uniformly for any $n\geq 1$, $x_\infty\in \partial \Omega_0^{a_i}$. Thus, there exists $k_0$ sufficiently large such that $u_{i,k_0}(x_\infty) \geq 3C$ for all $n\geq 1$. Since $u_{i,k_0}$ is uniformly continuous we deduce that
$|u_{i,k_0}(x_n)-u_{i,k_0}(x_\infty)|\to 0$. In other words,
 $$u_{i,k_0}(x_n)\geq u_{i,k_0}(x_\infty)-C\geq 3C,$$
 which is a contradiction since we were assuming that $u_{i,k}(x_n)\leq C$ for all $n\geq 1$ and $k\geq 1$.

Finally, to see that $U_{\lambda_\infty}$ is actually the minimal positive solution of \eqref{minL} we choose any solution $\widehat U_{\lambda_\infty}$ of \eqref{minL} and by
comparison,  $U_n < \widehat U_{\lambda_\infty}$ in $ \overline\Omega\setminus(\overline{\Omega}_0^{a_1}\cup \overline{\Omega}_0^{a_2}).$
Thus, letting $n\to \infty$ we deduce that $U_{\lambda_\infty} <\widehat U_{\lambda_\infty}$,
and, hence, $U_{\lambda_\infty}$ is the minimal positive solution of~\eqref{minL}.
\end{proof}

Finally we consider the non-symmetric case in which  the limit behaviour is given by $\lambda_\infty=\lambda^{m_1}[-\Delta,\Omega_0^{a_1}]<\lambda^{m_2}[-\Delta,\Omega_0^{a_2}]$. To do so, we consider again problem~\eqref{minL}--\eqref{minL_bdy}, but now coupled with the boundary condition
\begin{equation}
  \label{minL:bdy3}
  u_1=\infty\quad  \hbox{on}\ \partial \Omega_0^{a_1},\quad\hbox{and}\quad
   u_2=0\quad  \hbox{on}\ \partial \Omega_0^{a_2}.
\end{equation}

\begin{theorem}
\label{Th:exist_case1}
For any fixed $\lambda\in (\lambda^*,\lambda_\infty)$ let ${U}_\lambda$ be the unique positive solution of~\eqref{eq:main.compact}.
If $\lambda_\infty=\lambda^{m_1}[-\Delta,\Omega_0^{a_1}]<\lambda^{m_2}[-\Delta,\Omega_0^{a_2}]$, then
$U_{\lambda} \to {U}_{\lambda_\infty}$ uniformly on compact subsets of $\overline\Omega\setminus(\overline{\Omega}_0^{a_1}\cup \overline{\Omega}_0^{a_2})$,
as $\lambda\to \lambda_\infty$, where $U_{\lambda_\infty}$ is the minimal positive solution of~\eqref{minL}--\eqref{minL_bdy} and~\eqref{minL:bdy3}.
\end{theorem}

\begin{proof}
The proof follows the same idea as the one performed to prove Theorem~\ref{lamlinTh} and we omit it. We only want to remark that, when passing to the limit we get for $u_{2,\infty}$
$$
-\Delta u_2=\lambda^{m_1}[-\Delta,\Omega_0^{a_1}] m_2(x) u_2-a_2(x) u_2^p \quad   \hbox{in}\quad \Omega_2\setminus\overline\Omega_0^{a_2}.
$$
This problem has a unique positive solution, since $\lambda^{m_1}[-\Delta,\Omega_0^{a_1}]<\lambda^{m_2}[-\Delta,\Omega_0^{a_1}]$, see~\cite{DuHuang}.
\end{proof}

\section{Conclusions and further work}
\label{sect:conclusions}

In this work, we characterise the solutions of a steady-state model of population migration through a membrane in terms of the  intrinsic growth rate of the populations when crowding effects for those populations are considered. Our strategy relies on the previous analysis of related linear problems, some of them studied in~\cite{Suarezetal} and~\cite{WangSu} and applying afterward the sub and supersolutions technique,~\cite{AC-LG2}, and the results in~\cite{CR},~\cite{DuHuang} and~\cite{R1} to fully characterise the behaviour of positive solutions to the system.

A very interesting direction both from the biological
and mathematical point of view, could be coupling the system with a third equation, representing a population that inhabits everywhere and acts as either
prey or predator of the other two species. Moreover, we could also consider, following~\cite{Suarezetal}, different permeability conditions
on the membrane, not only from one side to the other side of the domain, i.e. non symmetric conditions, but also a permeability that depends on the region of crossing from one side to the other. We think that this would be a realistic approach to  model, for instance, problems assuming
 geographical barriers with different types of permeability terms.

Finally, our interest in analysing the existence of such stationary solutions comes from
the fact that this analysis is imperative as the first necessary step towards ascertaining
the dynamics of the associated parabolic problem. Although such a dynamical analysis has not been carried out in this work we plan to perform it shortly.

\vspace{0.4cm}

%%%%%%%%%%%%%%%%%%%%%%%%%%%%%%%%%%%%%%%%%%%%%%%%%%%%%%%%%%%%%%%%%%%%%
\noindent{\bf Acknowledgements}.
The authors would like to express their deepest gratitude to the reviewers of this work for all the invaluable suggestions and corrections that have significantly 
improved the  final outcome of this work.
%%%%%%%%%%%%%%%%%%%%%%%%%%%%%%%%%%%%%%%%%%%%%%%%%%%%%%%%%%%%%%%%%%%%%

%%%%%%%%%%%%%%%%%%%%%%%%%%%%%%%%%%%%%%%%%%%%%%%%%%%%%%%%%%%%%%%%

%%%%%%%%%%%%%%%%%%%%%%%%%%%%%%%%%%%%%%%%%%%%%%%%%%%%%%%%%%%%%%%%%

\end{document}